\documentclass[12pt]{article}
\usepackage{amsfonts, amssymb, amsmath, amsthm, mathrsfs}
\usepackage{multirow, tabularx}
\usepackage{graphicx}
\usepackage{pb-diagram}
\usepackage{textcomp}
\usepackage{wrapfig}
\usepackage{indentfirst}
\usepackage{float}
\usepackage{xcolor}
\usepackage{ulem}

\newtheorem{theorem}{\bf Theorem}[section]

\newtheorem{corollary}{\bf Corollary}[section]
\newtheorem{example}{\bf Example}[section]
\newtheorem{lemma}{\bf Lemma}[section]

\newtheorem{remark}{\bf Remark}[section]

\begin{document}
\begin{center}
An effective criterion for a stable factorisation of strictly nonsingular $2\times2$ matrix functions. Utilisation of the 
{\sf ExactMPF} package.
 
\end{center}
\begin{center}
 N.\,V. Adukova$^{1}$, V.\,M. Adukov$^{2}$,  G. Mishuris$^{1}$
\end{center}
\begin{center}
$^{1}$Department of Mathematics, Aberystwyth University,\\ SY23 3BZ Aberystwyth, UK\\
{$^{2}$Institute of Natural Sciences and Mathematics,\\ South Ural State University, 454080 Chelyabinsk, Russia\\
}
\end{center}

\begin{abstract}
In this paper, we propose a method to factorise of arbitrary strictly nonsingular $2\times 2$ matrix functions  allowing for stable factorisation. For this purpose, we utilise 
the {\sf ExactMPF} package working within the Maple environment previously developed 
by the authors and performing an {\it exact} factorisation of a nonsingular polynomial matrix function. 
A crucial point in the present analysis is the evaluation of a stability region of the canonical factorisation of the polynomial matrix functions. This, in turn, allows us to propose a sufficient condition for the given matrix function admitting stable factorisation.      
\end{abstract}

\vspace{0.7cm}
\section{ Introduction and outline of the main result} 

Factorisation of matrix functions  is a challenging task when it comes to its practical implementation. Indeed, any available numerical method is only justified under the condition that partial indices of the matrix function in question are stable {\cite{GK58}}. On the other hand, there is no general explicit criterion allowing to determine those indices {\cite{RM}}.

Factorisation of matrix functions plays the central role in various applications, e.g. integration of nonlinear differential equations by the inverse scattering method 
\cite{TF, AM}, in the theory of the Markushevich
problem on the unit circle \cite{Litvinchuk} and in scattering and diffraction of elastic waves in bodies with obstacles \cite{KAMR,Ab87, PA, K18, PKA}. 

%\textcolor{red}{Natalia tut zakomentirovano nizhe ewe paru ssylok} 

%https://doi.org/10.1016/0165-2125(87)90002-3
%https://www.sciencedirect.com/science/article/abs/pii/S0165212519303270?via%3Dihub
% https://doi.org/10.1137/17M1136304
% http://dx.doi.org/10.1098/rsta.2019.0241
% 

Most of the problems discussed in applications refers to vectorial Wiener--Hopf problems with $2\times 2$ matrix functions 
\begin{equation}\label{A(t)}
A(t) = 
\begin{pmatrix}
a_{11}(t) & a_{12}(t) \\
a_{21}(t) & a_{22}(t)
\end{pmatrix}.
\end{equation}
It is known \cite{CG, LS, GF71, Pr} that a continuous invertible matrix function $A(t)$ admits a {\it right Wiener--Hopf factorisation}
\begin{equation}
 \label{e-fact}
A(t)=A_{-}(t)D(t)A_{+}(t), \ t \in \mathbb T,  
\end{equation}
where $A_{+}(t)$ and $(A_{-}(t))$ are continuous matrix functions on
$\mathbb T$ that can be extended analytically to the domains
${\cal D}_{+} = \left\{ z \in \mathbb C :\ |z|<1 \right\}$ and 
${\cal D}_{-} =
\left\{z \in \mathbb C \cup \{\infty\} :\ |z|>1 \right\}$, and are
invertible in the respective domain. 
Factor $D(t)$ is the diagonal matrix function
$D(t)={\rm diag}[t^{\varkappa_{1}},\ t^{\varkappa_{2}}]$. Integers $\varkappa_1\leq \varkappa_2$ are called the
{\it right partial indices} of $A(t)$ and they are unique. In contrast, the factors are not unique. We discuss possible normalisation that guarantees the uniqueness in the following.

Factorisation~(\ref{e-fact}) is said
to be {\it canonical} if all partial indices are equal to zero ($D(t)=I_2$).  In virtue of the 
Gohberg-Krein-Bojarskii criterion \cite{GK58}, a nonsingular matrix function $A(t)$ admits the stable factorisation if and only if $\varkappa_2-\varkappa_1\leq 1$. 
For example, matrix function with positive definite
real component admits the canonical factorisation
\cite{GK58}, while a matrix function not
containing zero in its numerical range has the partial indices equal to each
other \cite{Sp74}. Hence, both classes admit stable factorisations.

We call a matrix function to be {\it strictly} nonsingular $2\times 2$ on the unit
circle $\mathbb T$ if the following conditions satisfy
\begin{equation}
a_{11}(t)\neq 0,\  \det A(t)\neq 0,\  t \in \mathbb T. \label{e-strnons}
\end{equation}
Note that matrix functions with positive definite real components and matrix
functions not containing zero in their numerical ranges satisfy
conditions (\ref{e-strnons}) automatically 
\cite[\textit{Chapter II, section 6}]{CG} and \cite[\textit{Chapter II, section 1.3}]{Pr}, but not vice versa.
Some results on the partial indices of strictly nonsingular matrices were obtain in 
\cite{Adukov95}. 
Recently, an effective algorithm to construct an approximate numerical factorisation of a matrix function being arbitrary close to a given $2\times2$ strictly nonsingular matrix function has been proposed in \cite{ES}.  However, to the best of the author\rq{}s knowledge, there is still no rigorous proof available to attribute the constructed approximation to the given matrix function.  

In this paper, we close this gap. 
Namely, we develop a method allowing to determine and prove whether a {\it strictly} nonsingular  $2\times 2$  matrix function possesses stable set of partial indices and simultaneously to perform its approximate factorisation. 
A crucial point of our analysis is a development of efficient condition allowing to determine the stability region (where the approximate matrix functions preserve their {\it stable} partial indices). 
As approximation, we use polynomial matrix functions. This allows us to utilise  
the {\sf ExactMPF} package working within the Maple Software environment \cite{Maple} developed 
by the authors in \cite{AAM22}. It is worth to note that the paper \cite{AAM22} was based on the method of essential polynomials \cite{Adukov92}. This powerful tool allows for simultaneous left and right side fuctorisation of the given analitical matrix function. The method, however, is involved being difficult for realisation for external users. For this reason, it was implemented in software as the {\sf ExactMPF} package working in Maple Software \cite{Maple} environment where the factorisation process is fully automated. {\sf ExactMPF} can be used as a tool for numerical experiments with matrix factorisation and in any applications requiring factorisation of polynomial matrix functions. The listing of the package can be found in the Supplementary Material of \cite{AAM22}.
Below we outline the main idea of our approach.

Let $A(t)$ be an arbitrary invertible matrix function from 
the Wiener algebra 
$W^{2\times 2}$ and ${\cal EF}\subset W^{2\times 2}$ a class of matrix functions admitting an explicit solution of the factorisation problem. By the explicit solution of the factorisation problem we understand a clearly defined algorithmic procedure that definitely terminates after a finite number of steps. Additionally, we note that if a) the input data belongs to the Gaussian field  $\mathbb{Q}(i)$ of complex rational numbers and b) all (finite) steps of the explicit algorithm can be performed in the rational arithmetic then we say that the problem can be solved exactly \cite{AAM22}. 
Examples of the class ${\cal EF}$ are a class of
Laurent matrix polynomials \cite{Adukov99, GLR}, a class of meromorphic matrix functions \cite{Adukov92}, or  
a class of triangular $2\times 2$ matrix functions defined in \cite{AMR20}.

Now, suppose we have approximated $A(t)$ by $\widetilde A(t)\in {\cal EF}$ such that $\|A-\widetilde A\|_W<\delta$ and $\widetilde A(t)$ admits
the canonical factorisation $\widetilde A(t)=\widetilde A_-(t)\widetilde A_+(t)$. 
%\textcolor{blue}{We assume that
%the estimate $\delta$ can be constructed in an explicit form. WHY ??? why it is not enough to have things %numerically??? } 
Since the canonical factorisation is stable, there exists
a neighbourhood of $\widetilde A(t)$ consisting of matrix functions admitting the canonical factorisation.  
Assume we can obtain an explicit estimate of the radius $R$ of such a neighbourhood 
as $R\leq \widetilde r_{cf}$, where the formula for $\widetilde r_{cf}$ will be discussed below. 
If $\delta<\widetilde r_{cf}$ then $A(t)$ also admits
the canonical factorisation (see Figure \ref{fig_expl}). Similar reasoning can be applied to a stable factorisation ($\rho_2-\rho_1=1$) where $\widetilde r_{cf}$ will have a different representation though.
Note that the both estimates depend on the normalisation of matrix functions during the approximation as well as the factorisation procedures. 

\begin{figure}[h]
\centering
\includegraphics[width=0.7\linewidth]{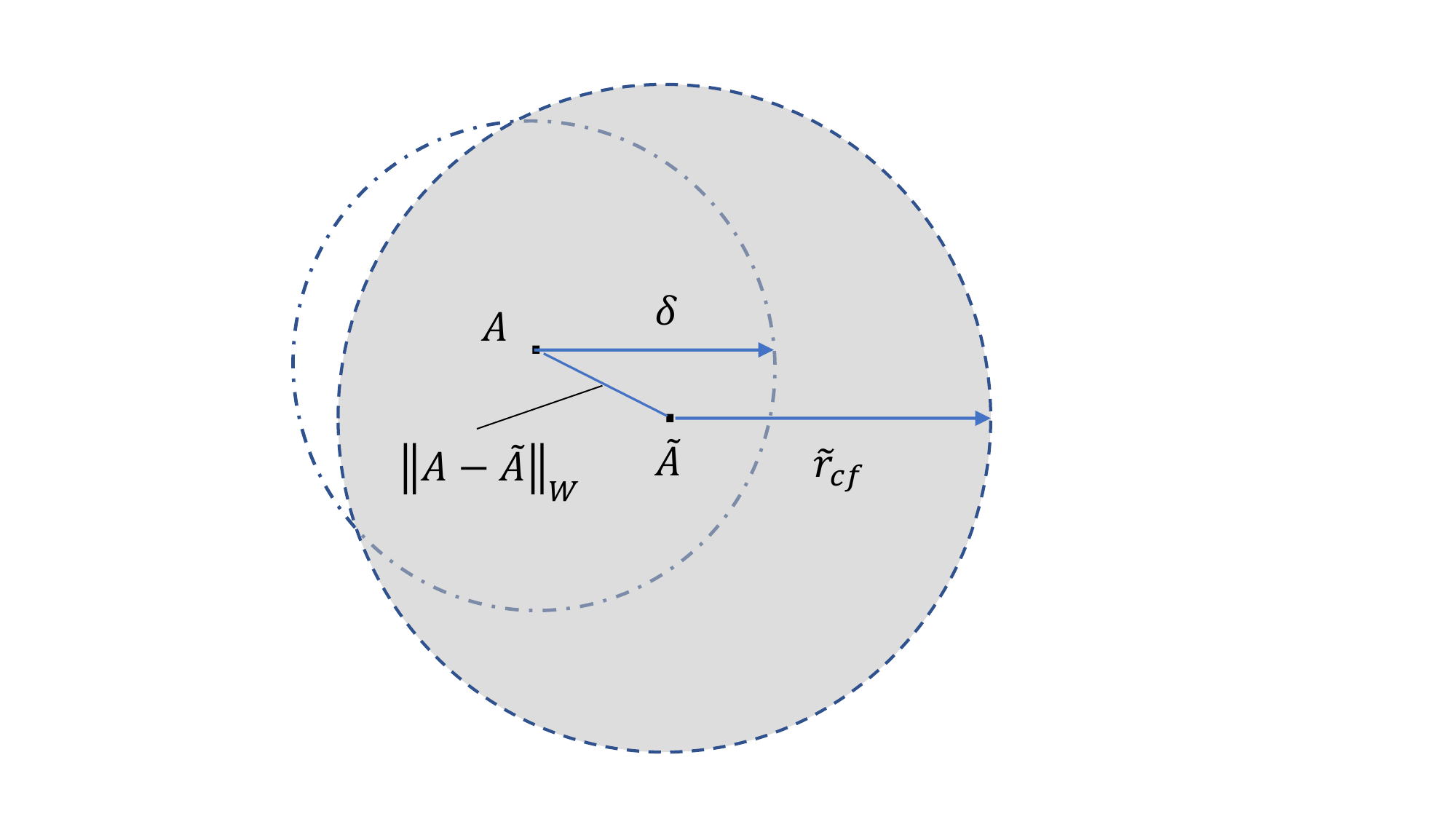}
\caption{Neighbourhoods of  $A$ and $\widetilde A$ for which  $A(t)$ expected to admit the canonical factorisation. If $\widetilde A$ admits a {\it explicit} canonical factorisation, then two inequalities:
$\| A-\widetilde A\|_W< \delta $ and $\delta <\tilde r_{cf}$ together guarantee that $A$ also possesses the same property. 
The goal of this paper is to deliver a tool to effectively verify the conditions.}
  \label{fig_expl}
\end{figure}

We note that factorisation of 
an arbitrary strictly nonsingular matrix function $A(t)$ from (\ref{A(t)}) can be explicitly reduced    
to factorisation an auxiliary (invertible) matrix function in the form (see section~\ref{pre} for details):
\begin{equation}
a(t) = \begin{pmatrix}\label{aform}
1 & \beta_{-}(t) \\
\alpha_{+}(t) & t^{\theta}+\alpha_{+}(t) \beta_{-}(t)
\end{pmatrix},
\end{equation}
where $\theta={\rm ind} \det a(t)= {\rm ind} \det A(t)-2\,{\rm ind}\,  a_{11}(t)$. 

The approach, highlighted above for matrix function $A(t)$ in its general form, can be applied to the matrix function $a(t)$. As a class $\cal EF$, we take Laurent matrix polynomials in the form:
\begin{equation}
\label{a_N}
a_N(t) = \begin{pmatrix}
1 & \beta_{-}^{(N)}(t) \\
\alpha_{+}^{(N)}(t) & t^{\theta}+\alpha_{+}^{(N)}(t) \beta_{-}^{(N)}(t) 
\end{pmatrix},
\end{equation}
where $\alpha_{+}^{(N)}(t)$, $\beta_{-}^{(N)}(t)$ are the Laurent polynomials of degree $N$ of the functions 
$\alpha_{+}(t)$, 
$\beta_{-}(t)$.

To obtain the explicit estimate of $\|a-a_N\|_W$,
we assume that $a(t)$ is analytical in some annulus containing
the unit circle $|t|=1$. In this case, the matrix Fourier coefficients of $a(t)$ coincide with the Laurent coefficients of $a(t)$ and we can use the Cauchy inequalities \cite{Ah} for the latter.

The factorisation of $a_N(t)$ can be explicitly constructed
 by the method of essential polynomials (see \cite{Adukov99, Adukov92, Adukov98}). Moreover, if the Laurent coefficients of $\alpha_+^{(N)}(t)$, $\beta_-^{(N)}(t)$ belong  to the field $\mathbb Q(i)$, then the factorisation problem for $a_N(t)$ can be solved {\it exactly} \cite{AAM22} with help of the package 
{\sf ExactMPF}. 
If the coefficients do not belong the field $\mathbb Q(i)$, we can find their rational approximations and to apply {\sf ExactMPF}. The use of the error-free calculations performed with the package is the  key idea of the approach as it guarantees exact computation of the partial indices. 

The paper is organised as follows. In the Section~\ref{pre}, we present important facts about the matrix Wiener algebra and the Teoplitz operators used in the following, discuss possible normalisation of the factorisation $A(t)$ and finally show the reduction of an arbitrary strictly nonsingular matrix function \eqref{A(t)} to \eqref{aform}.  In the section \ref{appro}, we give the approximation of the matrix function $a(t)$ and estimate its accuracy. Section \ref{cri} contains the main result of this article, namely a criterion of the stable factorisation for $a(t)$. In the section \ref{canon}, we obtain the explicit estimates for the absolute error in the approximate calculation of the factorisation factors. Finally, in the section \ref{exp}, we present some numerical results highlighting three different cases: canonical factorisation, a factorisation with the equal partial indices and a stable factorisation. Supplementary material collects full set of the data in form of tables.

\section {Preliminary considerations} 
\label{pre}

\subsection{Main definitions and known facts}
\label{main_def}

In this section, we provide some important definitions and follow the notation from \cite {CG, LS, GF71, GKS}.

Let $W^{p \times p}$ be the ${p\times p}$ matrix Wiener algebra consisting of matrix functions with entries from $W$.
Thus any $A(t)\in W^{p\times p}$ expands into  an absolutely convergent matrix Fourier series, $A(t)=\sum_{k=-\infty}^\infty A_kt^k$, such that $\sum_{k=-\infty}^\infty \|A_k\|<\infty$.
Here, $A_k$ belong to the algebra ${\mathbb C}^{p\times p}$ complex matrices   
equipped with the norm represented by any matrix multiplicative norm on
${\mathbb C}^{p\times p}$, preserving unity.
%Typically we will use the maximum column norm  $\|\cdot\|_1$ on ${\mathbb C}^{p\times p}$. 
Algebra
$W^{p\times p}$ becomes a Banach algebra if we endow it with a norm
$\|A\|_W=\sum_{k=-\infty}^\infty \|A_k\|_1$,
 where $\|\cdot\|_1$
is the maximum column sum matrix norm. Also we denote a group of invertible elements of the algebra $W^{p \times p}$ by $GW^{p \times p}$.

If
$A_k=(a_{ij}^{(k)})_{i,j=1,...,p}$, where $a_{ij}^{(k)}$ are the Fourier
coefficients of the element $a_{ij}(t)$ of the matrix function
$A(t)$, then it is easily seen that $\|A_k\|_1\leq
\sum_{i,j}|a_{ij}^{(k)}|$. Then $\|A\|_W\leq
\sum_{i,j}\|a_{ij}\|_W$.
Let us define 
\begin{multline*} W_{+}^{p\times p} = \biggl\{ A(t) \in W^{p\times p}: 
A(t) = \sum_{k=0}^{\infty}A_kt^k  \biggr\}, \\ 
W_{-}^{p\times p}= \biggl\{ A(t) \in W^{p\times p}: A(t) = \sum_{k=-\infty} ^{0}A_kt^k\biggr\},\\  
(W_{-}^{p\times p})_0 = \biggl\{ A(t) \in W^{p\times p}:
A(t) = \sum_{k=-\infty} ^{-1}A_kt^k\biggr\}. 
\end{multline*} 
It is known that
$W_{\pm}^{p\times p}$, $(W_-^{p\times p})_0$ are closed subalgebras of $W^{p\times p}$ and
$$
W^{p\times p} = W_{+}^{p\times p}\ \oplus\ (W_{-}^{p\times p})_0. 
$$
Thus $W_{\pm}^{p\times p}$ is a decomposing algebra.

By ${\cal P}_{+}$ we denote a projector from $W^{p\times p}$ onto 
$W_{+}^{p\times p}$ along $(W_{-}^{p\times p})_0$ and ${\cal P}_{-} = {\cal I}-{\cal P}_{+}$. Here ${\cal I}$ is the identity operator. 
Operator ${\cal P}_{+}$ acts according to the rule
$$
{\mathcal{P}}_{+} \sum\limits_{k=-\infty }^{\infty }{A_{k}{t}^{k}}=\sum\limits_{k=0}^{\infty }{A_{k}{t}^{k}}.
$$
and is a linear bounded operator in $W^{p\times p}$, while $\|\mathcal{P}_{+}\|= 1$.

Let $A(t)\in W^{p\times p}$. On the Banach space $W_{+}^{p\times p}$ we define the operator $T_{A}$ acting according to the rule
$$
T_{A} X(t)= \mathcal{P}_{+} A(t) X(t).
$$

It is obvious that $T_{A}$ is a linear bounded operator and 
$\|T_{A}\| \le \|A\|_{W}$. $T_A$ is called {\it the Toeplitz operator with the matrix symbol $A(t)$}.
Note that even though the Toeplitz operator is usually considered on the Banach space
 $W_{+}^{p\times 1}$, for us it is more convenient to operate on $W_{+}^{p\times p}$. 
It is straightforward to prove the {\it partial multiplicativity} property of the mapping $A\rightarrow T_A$, that is,
for any $A_{\pm }\in {W_{\pm }}^{p\times p}$ 
and $B\in W^{p\times p}$, the relations
$$
T_{B A_{+}}=T_{B} T_{A_{+}}, \ \  T_{A_{-} B}=T_{A_{-}} T_{B}
$$
are valid.
As a result, all properties of the standard Toeplitz operator are preserved. 

We will also need the following two statements.
 
If $A(t)=A_-(t)D(t)A_+(t)$ is the right Wiener--Hopf factorisation 
of an invertible matrix function $A(t)$, then the Toeplitz operator
$T_A$ is the right (left) invertible if and only if all right  partial
indices are non-positive (non-negative). In this case
$T_A^{(-1)}=T_{A_+^{-1}}T_{D^{-1}}T_{A_-^{-1}}$ is its one-sided inverse
\cite[ch.VIII, Cor. 4.1]{GF71}

If a linear bounded operator ${\cal A}$ is one-side invertible, and 
${\cal A}^{(-1)}$ is its one-side inverse, then any operator
$\widetilde{\cal A}$ satisfying the inequality 
\begin{equation}
\label{norm}
\|{\cal A}- \widetilde{\cal A}\|<\frac{1}{\|{\cal A}^{(-1)}\|}
\end{equation}
is also one-side invertible (from the same side as ${\cal A}$).
Moreover, if ${\cal A}$ is one-side invertible but not invertible, then
$\widetilde{\cal A}$ is also not invertible \cite[ch.II, Th. 5.4]{GK73}. Note that (\ref{norm}) is true in any elements of an abstract Banach algebra.

\subsection{Normalised factorisation of matrix functions}

Normalisation  of the factorisation plays a crucial role when performing factorisation numerically. This is specifically in case when there is a need to compare the consequent approximations.
Usually, a normalisation is chosen ad hock in a line with chosen numerical/asymptotic procedure without any justification and a success depends on each specific case (see for example \cite{RM14, RM16, RM18}). 

Unfortunately, being important, this issue still has not been fully resolved. The effective results are known for:
a) $2 \times 2$ matrix functions  \cite{Adukov22}; for matrix functions of arbitrary dimension $p \times p$ ($p>2$), when b) the set of partial indices are stable \cite{Adukova22} or c) all partial indexes are different \cite{Ad22}.

Since the paper deals with $2\times2$ matrices, we present here the normalisation only for such case. 
Two cases should be distinguished.
\begin{itemize}
\item
If  matrix function admits a factorisation  
$A(t)=t^{\varkappa_1}A_-(t)A_+(t)$, ($\varkappa_1=\varkappa_2$) then it can be 
trivially normalised by the condition $A_-(\infty)=I_2$, where $I_2$ is the unit matrix, and 
such factorisation is unique.
\item
In general case ($\varkappa_1<\varkappa_2$), there exists the so-called 
$P$-normalised factorisation of $A(t)$ that guarantees its uniqueness \cite{Adukov22}. The type of normalisation is determined by a $2\times 2$ permutation matrix $P$,
$P=I_2$ or 
$P=J_2:=\begin{pmatrix}
0&1\\
1&0
\end{pmatrix}$. 
\end{itemize}

The latter can be written in an explicit form: 
if $\bigl(A_-(\infty)\bigr)_{11}\ne 0$ 
(or $\bigl(A_-(\infty)\bigr)_{21}\ne 0$), then $A(t)$ admits the $I_2$-normalised (or $J_2$-normalised) factorisation \cite[Th. 2.1]{Adukov22}.

In a particular case of the stable factorisation ($\rho_2-\rho_1=1$) when $\bigl(A_-(\infty)\bigr)_{11}$ $\ne 0$, 
the $I_2$-normalisation is carried out as follows 
(see more details in \cite{Adukov22,Adukova22}). 
Let $A_-(t)=A_0+A_1 t^{-1}+\ldots$, and $A_0=L_0U_0$ is $LU$-factorisation of the limiting matrix function
$A_0=A_-(\infty)$. Define a matrix polynomial $Q_-(t) = Q_0 +Q_1t^{-1}$ such that  
\begin{equation}\label{Inorm} 
Q_0=U_0^{-1},\ \ Q_1=-\frac{1}{(A_0)_{11}}\begin{pmatrix}
0&\bigl(A_1U_0^{-1}\bigr)_{12}\\
0&0
\end{pmatrix}.
\end{equation}

Let $C_-(t)=A_-(t)Q_-(t)$, $C_+(t)=D^{-1}(t)Q_-^{-1}(t)D(t)A_+(t)$.
Then $A(t)=C_-(t)D(t)C_+(t)$ is the sought for $I_2$-normalised factorisation of $A(t)$.
$J_2$-normalised factorisation constructed analogously by the $JLU$-factori\-sa\-tion of $A_0$.

Finally, $P$-normalisation is stable under a small perturbation $A(t)$, that is, for arbitrary sufficiently
small $\delta>0$, each matrix function $\widetilde A(t)$ possessing the same set of the right partial indices $\rho_1,\rho_2$, as $A(t)$ and satisfying the inequality $\|A-\widetilde A\|_W<\delta$ has the same type of the $P$-normalisation as $A(t)$ \cite[Th. 4.1]{Adukov22}.

\subsection{Reduction of matrix function $A(t)$ to the form $a(t)$}
\label{represent}

In this subsection, we deliver a representation of the matrix function $A(t)$ given in the form  \eqref{A(t)} allowing to 
consider a simpler matrix function, $a(t)$, in the form \eqref{aform}.
Let
\begin{equation}
\label{minors}
a_{11}(t) = a_{11}^{-}(t)t^{\varkappa}a_{11}^{+}(t), \ \
\Delta(t) =\Delta_{-}(t)t^{\theta+2\varkappa}\Delta_{+}(t) 
\end{equation} 
be the
Wiener--Hopf factorisations of the functions $a_{11}(t)$ and
$\Delta(t) = \det A(t)$, respectively. 
Not that the indices $\varkappa$ and $\theta+2\varkappa$ in \eqref{minors} do not coincide, generally speaking, with the pair $\varkappa_1$, $\varkappa_1+\varkappa_2$ (compare \eqref{e-fact}). The only relationship between them is: 
$\varkappa_1+\varkappa_2=\theta+2\varkappa$. 

The matrix function $A(t)$ can be equivalently written in the form: 
\begin{equation}
\label{rep1}
A(t)=
t^{\varkappa}
\begin{pmatrix}
a_{11}^-(t)&0\\
0&\frac{\Delta_-(t)}{a_{11}^-(t)}
\end{pmatrix}
\begin{pmatrix}
1&\beta(t)\\
\alpha(t)&\frac{a_{22}a_{11}^-a_{11}^+}{t^\varkappa\Delta_-\Delta_+}\\
\end{pmatrix}
\begin{pmatrix}
a_{11}^+(t)&0\\
0&\frac{\Delta_+(t)}{a_{11}^+(t)}
\end{pmatrix},
\end{equation}
where
\begin{equation}
\label{coefficients} 
\alpha (t) = \frac{a_{11}^{-}(t)a_{21}(t)}
{\Delta_{-}(t)t^{\varkappa}a_{11}^+(t)}, \ \ \beta (t) =
\frac{a_{11}^{+}(t)a_{12}(t)} {\Delta_{+}(t)t^{\varkappa}a_{11}^{-}(t)}.
\end{equation}

Introducing functions $\alpha_{\pm}(t) = \mathcal{P}_{\pm}\alpha (t)$,
$\beta_{\pm}(t)=\mathcal{P}_{\pm}\beta(t)$,
matrix function $A(t)$ can be further transformed to the form:

\begin{equation}
\label{final}
 A(t)=
 t^{\varkappa}\begin{pmatrix}
a_{11}^{-}(t) & 0 \\
 & \\
\frac{\Delta_{-}(t)\alpha_{-}(t)}{a_{11}^{-}(t)} &
\frac{\Delta_{-}(t)}{a_{11}^{-}(t)} 
\end{pmatrix}
a(t)
\begin{pmatrix}
a_{11}^{+}(t) & \frac{\Delta_{+}(t)\beta_{+}(t)}{a_{11}^{+}(t)} \\
 & \\
0 & \frac{\Delta_{+}(t)}{a_{11}^{+}(t)} 
\end{pmatrix},
\end{equation}
where
\begin{equation}\label{matrix_a}
a(t) =
\begin{pmatrix}
1&0\\
-\alpha_-(t)&1
\end{pmatrix}
\begin{pmatrix}
1&\beta(t)\\
\alpha(t)&\frac{a_{22}a_{11}^-a_{11}^+}{t^\varkappa\Delta_-\Delta_+}\\
\end{pmatrix}
 \begin{pmatrix}
1&-\beta_+(t)\\
0&1
\end{pmatrix}.
\end{equation}
It remains to observe that this matrix function has the form \eqref{aform}:
\begin{equation}
a(t) = \begin{pmatrix}\label{aform_doc}
1 & \beta_{-}(t) \\
\alpha_{+}(t) & t^{\theta}+\alpha_{+}(t) \beta_{-}(t)
\end{pmatrix}.
\end{equation}

For 
$\theta =0$ this reduction was carried out in \cite{Adukov95} where an explicit formulas for the partial indices of $a(t)$ were obtained under additional condition that $\alpha_+(t)$ is
a polynomial in $t$ or $\beta_-(t)$ is
a polynomial in $t^{-1}$.

%\textcolor{red}{ya dumayu chto v etom meste nado skazat\rq{} chto matrica funkciya udovletvoryaet usloviyu 
%$\left(a(t)\right)_{11}=1\ne0$ a znachit ona P-factorizema} \textcolor{green}{ \underline{Comment:} I'm not sure, but it seems to me that, in order to normalise, we firstly need to construct a factorisation}

Note that the matrix function $a(t)$ is invertible ($\det a(t)=t^\theta$) and thus can be factorised: 
\begin{equation}
\label{a_fact}
a(t)=a_-(t)d_r(t)a_+(t),\quad  d_r(t)={\rm diag}[t^{\rho_1}, t^{\rho_2}],
\end{equation}
where $\rho_1\le \rho_2$ and $\theta=\rho_1+\rho_2$.

Comparing the latter with \eqref{a_fact} and \eqref{final} we conclude that the partial indices of the matrices $A(t)$ and $a(t)$ are related in the following manner: 
\begin{equation}
\label{indices_both}
\varkappa_1=\varkappa+\rho_1,\quad \varkappa_2=\varkappa+\rho_2.
\end{equation}

In the following, we will focus on the stable factorisation of the matrix function $a(t)$. Hence, 
\begin{equation}
\label{stable}
\begin{matrix}
\rho_1=\rho_2= \nu \quad {\rm for \,\, even}\quad \theta=2\nu, \\[2mm]
\rho_1=\nu , \rho_2= \nu+1\quad {\rm  for \,\, odd}\quad  \theta=2\nu+1.
\end{matrix}
\end{equation}

As we have already mentioned above, when the authors of \cite{ES} built  their approximation, they used a completely different representation of the matrix function $A(t)$, in comparison to \eqref{final}:

%\textcolor{blue}{  
%Let 
%$$
%S(t)=

\begin{equation*}
 S(t)=
 \begin{pmatrix}
s_{11}^{+}(t) & 0 \\
 & \\
\frac{t^{-k_1} s_{21}(t)}{s_{11}^{-}(t)} &
\frac{\Delta^{+}(t)}{s_{11}^{+}(t)} 
\end{pmatrix}
 \begin{pmatrix}
 t^{k_1} & 0 \\
 0 & t^{k-k_1}
 \end{pmatrix}
\begin{pmatrix}
s_{11}^{-}(t) & \frac{t^{-k_1}s_{12}(t)}{s_{11}^{+}(t)} \\
 & \\
0 & \frac{\Delta^{-}(t)}{s_{11}^{-}(t)} 
\end{pmatrix},
\end{equation*}
where $s_{11}(t)=s_{11}^{+}(t)t^{k_1}s_{11}^{-}$ and $\det{S(t)}=\Delta^{+}(t)t^{k}\Delta^{-}(t)$.

\section{Approximation of the matrix function $a(t)$ by the Laurent matrix polynomial $a_N(t)$}\label{appro}

In this section we approximate $a(t)$ by the Laurent matrix polynomial $a_N(t)$ and estimate the norm of the difference $\|a-a_N\|_W$.

Let $\alpha_+(t)=\sum_{k=0}^\infty \alpha_k t^k$ and
$\beta_-(t)=\sum_{k=-\infty}^{-1}\beta_kt^k$.  We denote
$$\alpha_+^{(N)}(t)=\sum_{k=0}^N\alpha_kt^k, \ \
\beta_-^{(N)}(t)=\sum_{k=-N}^{-1}\beta_kt^k
$$ 
and consider the Laurent
matrix polynomial defined in \eqref{a_N}.

Similarly to the matrix function $a(t)$, since $\det a_N(t)= t^{\theta}$, then $a_N(t)$ is an invertible
element of $W^{2\times 2}$ for any $N$.

Now we can estimate the norm $\|a-a_N\|_W$ as
\begin{multline*}
\|a-a_N\|_W\leq
\|\alpha_+-\alpha_+^{(N)}\|_W+\|\beta_{-}-\beta_-^{(N)}\|_W+ \|\alpha_+\beta_{-}-\alpha_+^{(N)}\beta_-^{(N)}\|_W.
\end{multline*}
Representing
$$\alpha_+(t)\beta_-(t)-\alpha_+^{(N)}(t)\beta_-^{(N)}=\alpha_+(t)\Bigl(\beta_-(t)-\beta_-^{(N)}(t)\Bigr)+
\beta_-^{(N)}(t)\Bigl(\alpha_+(t)-\alpha_+^{(N)}(t)\Bigr)
$$
and taking into account that $\|\beta_-^{(N)}\|_W\leq
\|\beta_{-}\|_W$, we get
\begin{multline}
\label{estaaN1}
\|a-a_{N}\|_W\leq\\
\Bigl(1+\|\alpha_+\|_W\Bigr)\|\beta_{-}-\beta_-^{(N)}\|_W+\Bigl(1+\|\beta_-\|_W\Bigr)
\|\alpha_{+}-\alpha_+^{(N)}\|_W.
\end{multline}
The last inequality gives us a qualified estimate of the norm of the difference between the matrix functions  $a_N(t)$ and $a(t)$ defined in \eqref{aform} and \eqref{a_N}, respectively.  

To evaluate explicitly the convergence rate, we make additional assumptions on the rate of decay of the Fourier coefficients. 
Namely, we restrict ourselves to the case when the matrix function $a(t)$ is analytic 
in the annulus $r_1<|t|<r_2$ containing the unit circle, $\mathbb T$, while $r_1$ and $r_2$ can be zero and infinity, respectively.
%$0\leq r_1<1<r_2\leq\infty$.}
This effectively means that the function $\alpha_+(t)$ is analytic in the
domain $|t|<r_2$ while $\beta_-(t)$ is analytic in the domain 
$|t|>r_1$.

Obviously, the restriction of $a(t)$ on the unit circle
$|t|=1$ belongs to the Wiener algebra $W^{2\times 2}$ and
its matrix Fourier coefficients $a_k$ coincides with the coefficients 
of the Laurent series for $a(t)$ in the annulus $r_1<|t|<r_2$.

The
Laurent coefficients $|\alpha_k|$ of the function $\alpha_+(t)$ hold  
Cauchy's inequalities \cite{Ah}:
\begin{equation}
\label{alpha_est}
|\alpha_k|\leq\frac{M_+(\zeta)}{\zeta^k}
\end{equation}
for any $\zeta$, $0<\zeta<r_2$. Here $M_+(\zeta)=\max_{|t|=\zeta}|\alpha_+(t)|$. 

Fixing a value $\zeta=\zeta_2$, $(1<\zeta_2<r_2)$, 
we can estimate the sums of series $\sum_{k=0}^\infty |\alpha_k|$,
$\sum_{k=N+1}^\infty |\alpha_k|$ and obtain the following
inequalities:

\begin{equation}
\label{alpha_norm}
\|\alpha_+\|_W\leq\frac{\zeta_2M_+(\zeta_2)}{\zeta_2-1}, \ \  \|\alpha_{+}-\alpha_+^{(N)}\|_W
\leq\frac{M_+(\zeta_2)}{\zeta_2^N(\zeta_2-1)}, M_+(\zeta_2)=\max_{|t|\leq \zeta_2}|\alpha_+(t)|.
\end{equation}

%where and $M_+(\zeta_2)=\max_{|t|\leq \zeta_2}|\alpha_+(t)|$.
If $r_2<\infty$, and $\alpha_+(t)$ is continuously extended on  
$|t|=r_2$, then the Cauchy inequalities hold for $\zeta=r_2$ and, in this case, we can take also $\zeta_2=r_2$.

For the function $\beta_-(t)$,
we can obtain similar estimates fixing a value $\zeta=\zeta_1$, 
($r_1<\zeta_1<1$):
\begin{equation}
\label{beta_norm}
\|\beta_-\|_W\leq\frac{\zeta_1M_-(\zeta_1)}{1-\zeta_1}, \ \  \|\beta_{-}-\beta_-^{(N)}\|_W
\leq\frac{\zeta_1^{N+1}M_-(\zeta_1)}{(1-\zeta_1)},
\end{equation}
where $M_-(\zeta_1)=\max_{|t|=\zeta_1}|\beta_-(t)|=
\max_{|t|\geq\zeta_1}|\beta_-(t)|$. If $r_1>0$ and $\beta_-(t)$ admits a continuous extension on 
the circle $|t|=r_1$, the value $\zeta_1=r_1$ is also admissible. 

This allows us to obtain an explicit estimate (\ref{estaaN1}) 
\begin{equation}
\label{estaaN2}
\|a-a_{N}\|_W\leq\delta_N(\zeta_1,\zeta_2),
\end{equation}
where $r_1<\zeta_1<1$, $1<\zeta_2<r_2$ and
\begin{equation}
\label{delta_N12}
\delta_N(\zeta_1,\zeta_2)=\Bigl(1+\frac{\zeta_2M_+(\zeta_2)}{\zeta_2-1}\Bigr)
\frac{\zeta_1^{N+1}M_-(\zeta_1)}{1-\zeta_1}+\Bigl(1+\frac{\zeta_1M_-(\zeta_1)}{1-\zeta_1}\Bigr)\frac{M_+(\zeta_2)}{\zeta_2^N(\zeta_2-1)}.
\end{equation}
The values $\zeta_1=r_1$ and $\zeta_2=r_2$
can be also admissible. 

It is obvious that $\delta_N(\zeta_1,\zeta_2)$
is monotonic decreasing as $N\to \infty$ 
and tends to zero for any fixed $\zeta_1, \zeta_2$.
Although,  if $\zeta_1$, $\zeta_2$ are close enough to $1$, then the convergence is slow.  
On the other hand, the function increase when $\zeta_1$ or $\zeta_2$ are close to the other ends of their intervals.
Therefore, it is desirable to optimize the choice $\zeta_1$, $\zeta_2$ minimizing the function
$\delta_N(\zeta_1,\zeta_2)$ on the rectangle 
$(r_1,1)\times (1, r_2)$. Depending on whether the series converge on the closed or open domain, we can distinguish four respective cases.  

\begin{enumerate}
\item $r_1>0, r_2 <\infty$, that is, the function $\alpha_+(t)$ is analytic into $|t|< r_2$, 
and $\beta_-(t)$ is analytic into $|t|> r_1$.
\item $r_1>0, r_2 =\infty$, that is, the function $\alpha_+(t)$
is an entire function, 
and $\beta_-(t)$ is analytic into $|t|> r_1$.
\item $r_1=0, r_2 <\infty$, that is, the function $\alpha_+(t)$ is analytic into $|t|< r_2$, 
and $\beta_-(t)$ is analytic into $\overline{\mathbb C}\setminus \{0\}$.
\item $r_1=0, r_2 =\infty$, that is, the function $\alpha_+(t)$
is an entire function, 
and $\beta_-(t)$ is analytic into $\overline{\mathbb C}\setminus \{0\}$.
\end{enumerate} 
In order to effectively use \eqref{delta_N12},  we need to estimate $M_-(\zeta_1)$, 
$M_+(\zeta_2)$ in the respective domains with the best possible accuracy.
Here a difficulty steams from the fact that the minimization is sought in the open set $(r_1,1)\times (1, r_2)$.
In practice we restrict ourselves to a closed rectangular $K$ embedded into $(r_1,1)\times (1, r_2)$ that provides quite a reasonable approximation.

\section{A criterion of the stable factorisation for the matrix function $a(t)$}\label{cri}

Having constructed a Laurent matrix polynomial (\ref{a_N}) of degree $N$, that is an approximant of the matrix function (\ref{aform}) allowing for a stable factorisation

as is described in the section \ref{appro} above, we can compute the following measure:
\begin{equation}
\label{q_n}
q_N(\zeta_1,\zeta_2)=\sigma\,\delta_N(\zeta_1,\zeta_2)\|\bigl(a_+^{(N)}\bigr)^{-1}\|_W\,\|\bigl(a_-^{(N)}\bigr)^{-1}\|_W,
\end{equation}
where
\begin{equation}
\label{sigma}
\sigma=\|d_r\|_W=\|d_r^{-1}\|_W=
\begin{cases}
1,&\text{if $\theta=2\nu$},\\
2,&\text{if $\theta=2\nu+1$.}
\end{cases}
\end{equation}

\begin{theorem}\label{basethm} Let $a(t)$ admits a stable factorisation \eqref{a_fact}
and the factorisation is $P$-normalised. Then, there exists a natural $N_0$ such that, for $N>N_0$
and for all admissible pairs $(\zeta_1, \zeta_2)$, the following conditions are fulfilled
\begin{enumerate}
\item $a_N(t)$ admits a stable $P$-normalised factorisation: 
\begin{equation}
\label{fact_a_N}
a_N(t)=a_-^{(N)}(t)d_r(t)a_+^{(N)}(t),
\end{equation}
%and this factorisation is $P$-normalized.
\item $q_N(\zeta_1,\zeta_2)<1$.
\end{enumerate}
\end{theorem}

\begin{proof} 
Since the sequence
$a_N(t)$ converges to the matrix function $a(t)$ admitting a stable factorisation, it is always possible to find a natural $N_0$ such that, for all $N>N_0$, $a_N(t)$ admits the same factorisation (the same set of partial indices) \cite{GK58}. However, to deliver an estimate for the stability domain, more accurate analysis is required. This is the aim of this theorem.

Since $a_N(t)$ converges to the matrix function $a(t)$ 
there exists an integer $N_0$ such that 
\begin{equation}\label{vici}
\|a(t)-a_{N}\|_W <\frac{1}{\sigma\|a_+^{-1}\|_W\|a_-^{-1}\|_W}, \ \ \text{if} \ \ N\geq N_{0}.
\end{equation}
From (\ref{vici}) and \eqref{sigma}, it follows that
$$
\|a-a_{N}\|_W < \frac{1}{\|a^{-1}\|_W}, \ \ \text{if} \ \ N\geq N_{0}.
$$
The latter means that $a_N(t)$ is an invertible element of $W^{2 \times 2}$ and thus, it admits a factorisation $a_N(t)=a_-^{(N)}(t)d_r^{(N)}(t)a_+^{(N)}(t)$. 
The next step in the proof is to demonstrate that the latter inequality guarantees also that $a_N(t)$ belongs to the stability region (or equivalently, $d_r^{(N)}(t)=d_r(t))$.

By the assumption of the theorem, the matrix function $a(t)$ admits the stable factorisation $a(t)=a_-(t)d_r(t)a_+(t)$, that is 
$d_r(t)=t^\nu I_2$ if $\theta=2\nu$, and 
$d_r(t)=t^\nu \begin{pmatrix}
1&0\\
0&t
\end{pmatrix}
$ for $\theta=2\nu+1$. Moreover, since we consider the right partial indices of the matrix function $a(t)$, they are  both non-negative for $\nu \geq 0$ and non-positive for $\nu < 0$. We will demonstrate now that $d_r^{(N)}(t)=d_r(t)$.

Let us consider Toeplitz operators  $T_a$ and $T_{a_{N}}$ with the symbols $a(t)$ and
$a_{N}(t)$, respectively. Since the partial indices have the same sign, then the operator $T_a$ is one-sided invertible 
(invertible if $\nu=0$ ) and 
$T^{(-1)}_a=T_{a_+^{-1}}T_{d_r^{-1}}
T_{a_-^{-1}}$ is its one-sided inverse (inverse in the case $\nu=0$), moreover (see \cite{GF71})
$$\|T^{(-1)}_a\|\leq 
\|T_{a_+^{-1}}\|\|T_{d_r^{-1}}\|\|T_{a_-^{-1}}\|
\leq \sigma{\|a_+^{-1}\|_W\|a_-^{-1}\|_W}.
$$
Then, from (\ref{vici}), we have
$$
\|T_a-T_{a_{N}}\| \leq \|a-a_{N}\|_W<
\frac{1}{\sigma\|a_+^{-1}\|_W\|a_-^{-1}\|_W}\leq\frac{1}{\|T^{(-1)}_a\|}.
$$
This in turn means that the operator $T_{a_{N}}$ is also one-side invertible (or invertible if $\nu =0$). Moreover, the operators $T_{a}$ and $T_{a_{N}}$ are invertible from the same side and, therefore, in the factorisation $a_{{N}}(t)=a_-^{(N)}(t)d_r^{(N)}(t)a_+^{(N)}(t)$, $d_r^{(N)}(t)=
{\rm diag}[t^{{\widetilde\rho_1}}, t^{{\widetilde\rho_2}}]$, indices 
$\widetilde\rho_1, \widetilde\rho_2$ have the same signs as 
$\rho_1, \rho_2$.

The following three cases are possible: \\
1) If $a(t)$ admits the canonical factorisation, i.e. $\nu=0$, then $T_{a_{N}}$ is invertible and, hence, $a_{N}(t)$ admits also the canonical factorisation.
\\
2) Let $a(t)$ admits a stable factorisation and $\theta=2\nu$, that is 
$d_r(t)=t^\nu I_2$. Consider auxiliary matrix functions $b(t)=t^{-\nu}a(t)$
and $c(t)=t^{-\nu}a_N(t)$. By the construction, $b(t)$ admits a canonical factorisation,
$b(t)=a_-(t)a_+(t)$. On the other hand, since $\|t^{\nu}\|_W=\|t^{-\nu}\|_W=1$ and, due to inequality~(\ref{vici}), we conclude
\[
\|b-c\|_W<\frac{1}{\|b_+^{-1}\|_W\|b_-^{-1}\|_W}.
\]
Hence, (see the previous case), $c(t)$ also admits a canonical factorisation, $a_N(t)=t^{\nu}c(t)$ and $d_r^{(N)}(t)=d_r(t)$.
\\
3) Finally, let $a(t)$ admits a stable factorisation and $\theta=2\nu+1$, i.e.
$d_r(t)=t^\nu \begin{pmatrix}
1&0\\
0&t
\end{pmatrix}
$. 
Consider again the matrix functions $b(t)=t^{-\nu}a(t)$
and $c(t)=t^{-\nu}a_N(t)$. The auxiliary matrix function  
$b(t)=a_-(t)\begin{pmatrix}
1&0\\
0&t
\end{pmatrix} a_+(t)$ admits a stable factorisation, and 
\[
\|b-c\|_W<\frac{1}{2\|b_+^{-1}\|_W\|b_-^{-1}\|_W}.
\]
This proves that the respective operators are invertible from the same side, and their right indices are non negative. Thus, we conclude that
$\widetilde\rho_1\geq \nu$ and $\widetilde\rho_2> \nu$, and, since
$\widetilde\rho_1+\widetilde\rho_2=2\nu+1$, this is possible only if
$\widetilde\rho_1=\nu, \widetilde\rho_2=\nu+1$. This finishes the proof.

Summarising,  we have shown that, under condition \eqref{vici} the approximate matrix function $a_{N}(t)$ admits a stable factorisation with the same set of stable partial indices as the original matrix function $a(t)$. 

It remains to prove the second statement of the Theorem. Let $N\geq N_0$ and 
$a_N(t)=a_-^{(N)}(t)d_r(t)a_+^{(N)}(t)$ is now the stable $P$-normalised factorisation of $a_N(t)$. 
It follows from Th. 5.1 of \cite{Adukov22} that the factors  
$a_\pm^{(N)}(t)$ converge to normalised $a_\pm(t)$ as $N\to\infty$. Hence, the sequence
$\|\bigl(a_+^{(N)}\bigr)^{-1}\|_W\cdot
\|\bigl(a_-^{(N)}\bigr)^{-1}\|_W$ is bounded. %\textcolor{red}{Hence, sequence is converge to $a_{\pm}^{-1}(t)$ and is bounded \cite{AD} and also see the next section for the formula (\ref{est1plus}) in Th.\ref{invplus}.}
As we noted in the previous section, the sequence $\delta_N(\zeta_1,\zeta_2)$ monotonically tends to zero for any fixed pair $(\zeta_1,\zeta_2)$. Hence, starting from some $N$ the inequality $q_N(\zeta_1,\zeta_2)<1$ is fulfilled.
\end{proof}

\begin{theorem}\label{basethm1}
If conditions 1)--2) are fulfilled for at least one $N$ and one admissible pair $(\zeta_1, \zeta_2)$, then $a(t)$ admits the stable factorisation with the same set of partial indices as $a_N(t)$.
\end{theorem}

\begin{proof}
Since $\det a_N(t)=t^\theta$, $a_N(t)$ is invertible in the algebra $W^{2\times 2}$ for any $N$. Suppose that there exists $N$ such that $a_N(t)$ admits a stable $P$-normalised factorisation 
$
a_N(t)=a_-^{(N)}(t)d_r(t)a_+^{(N)}(t),
$
and for some 
$(\zeta_1,\zeta_2)$ it is fulfilled the inequality
\begin{equation}
\label{second_ineq}
\delta_N(\zeta_1,\zeta_2)<\frac{1}{\sigma\|\bigl(a_+^{(N)}\bigr)^{-1}\|_W\|\bigl(a_-^{(N)}\bigr)^{-1}\|_W},
\end{equation}
that coincides with $q_N<1$ (compare \eqref{q_n})
\begin{equation}\label{baseinq}
\|a-a_N\|_W\leq\delta_N(\zeta_1,\zeta_2)<\frac{1}{\sigma\|\bigl(a_+^{(N)}\bigr)^{-1}\|_W\,\|\bigl(a_-^{(N)}\bigr)^{-1}\|_W}.
\end{equation}
Now  we carry out the same reasoning as in the proof of the Theorem 4.1, only interchanging
$a(t)$ by $a_N(t)$. Thus, in virtue of (\ref{baseinq})
we have 
\begin{equation}\label{iqnnorm}
\|T_a-T_{a_N}\|\leq \|a-a_N\|_W<\frac{1}{\|\bigl(a_+^{(N)}\bigr)^{-1}\|_W\|d_r(t)\|_W\|\bigl(a_-^{(N)}\bigr)^{-1}\|_W}. 
\end{equation}
Since $a_N(t)$ admits the stable factorisation, the operator $T_{a_N}$ is one-side invertible, and
$T_{a_N}^{(-1)}=T_{(a_+^{(N)})^{-1}}T_{d_r^{-1}}T_{(a_-^{(N)})^{-1}}$ is its one-side inverse.

Now from (\ref{iqnnorm}) it follows that
$
\|T_a-T_{a_N}\|<\frac{1}{\|T_{a_N}^{(-1)}\|}.
$
This, in turn, means that the operator $T_a$ is one-side invertible. The same arguments as above show that the matrix function $a(t)$ admits the stable factorisation.
\end{proof}

\begin{remark}\label{anyfact}
The normalisation of the factorisation is used only in the proof of
necessity of conditions 1)-2). In fact, in the Theorem \ref{basethm} we can omit the $P$-normalisation requirement and require that the sequence $\|\bigl(a_+^{(N)}\bigr)^{-1}\|_W$ $\|\bigl(a_-^{(N)}\bigr)^{-1}\|_W$ is bounded from above. In the Theorem \ref{basethm1} we can use any factorisation. This allows us to avoid the computational difficulties associated with the procedure of $P$-normalisation. However, 
$P$-normalisation is a mandatory step in constructing an approximate factorisation with the guaranteed accuracy. 
\end{remark}
\begin{remark}\label{noncanonic}
Those theorems allow us to formulate a criterion for the case when matrix function $a(t)$ does not admit a stable factorisation. However, this would require checking an infinite number of conditions (for each $N$) and, thus, is unlikely to be implemented in practice.
\end{remark}

In fact, the results provided in the Sections~\ref{appro} and \ref{cri} allow to address the main challenge of this paper, providing two explicit conditions $\| a-a_N\|_W<\delta_N(\zeta_1,\zeta_2)$ and $q_N<1$ that a) can be numerically verified and b) guarantee that matrix function $a(t)$ has the same set of stable partial indices as $a_N(t)$.
To highlight things even further, we note that they represent those inequalities,
$\| a-a_N\|_W<\delta$ and $\delta<\tilde r_{cf}$, discussed around the Figure~\ref{fig_expl}.
Indeed, the inequality  $q_N<1$ can be equivalently written in the form \eqref{second_ineq} where the right-hand side corresponds to the stability radius $\tilde r_{sf}$.
The conditions are sufficient for stability the numerical procedure used in the approximate factorisation while continuity of the factors has been proven in \cite[Th. 5.1]{Adukov22}. However, the quality of the approximation remains unknown  and thus, if the procedure does not converge fast enough, this might be a serious obstacle for computations. 

Having a solid theoretical estimate for the convergence rate is always a benefit. Next section is devoted to this task for a classic canonical case ($\varkappa_1=\varkappa_2=0$) and a generalised canonical one ($\varkappa_1=\varkappa_2$). Unfortunately, the remaining stable case ($\theta= 2\nu+1$) requires more delicate treatment.

\section{Accuracy estimate of the factors in approximation of a canonical factorisation}
\label{canon}

In the first subsection, we briefly recall results of the work \cite{AD} on the continuity of factors $A_\pm(t)$ of the classic canonical factorisation (all partial indices are zero) and improve explicit estimates for the norm $\|A_\pm-A_\pm^{(N)}\|_W$ presented therein. Bearing in mind possible applications, we provide the analysis for matrix functions of an arbitrary order $p\ge2$.

\subsection{Explicit estimates for the factors of canonical factorisation of the original matrix function $A(t)$.}

We will assume that the canonical factorisation $A(t)=A_{-}(t)A_{+}(t)$ is normalised by the condition 
$A_{-}(\infty)=A_{0}$,
where $A_0$ is an arbitrary, predetermined, invertible matrix. This normalisation guarantees the uniqueness of the canonical factorisation. Clearly, we can choose a unit matrix as the limiting value for the factor ($A_0=I_p$), however, since the proof is not much dependent in this case of the normalisation we stay with more general formulation.

We need the following elementary lemmas. Their proofs for $q=1/2$ are given in 
\cite{AD}. Even the proofs do not differ much in the general case, the new results represent an improvement allowing to deliver the estimates in more convenient for practical use. 

\begin{lemma}\label{lemxx}
Let $\mathcal{A}$  be an invertible bounded linear operator in some Banach space and
and the operator $\widetilde{\mathcal{A}}$ satisfies the inequality
\begin{equation}
\label{morestrict}
\| \mathcal{A}-\widetilde{\mathcal{A}} \| \leq \frac{q}{ \| {\mathcal{A}}^{-1} \|}, \quad 0<q<1.
\end{equation} 
Then $\widetilde{\mathcal{A}}$ is also invertible and, for solutions of the equations $\mathcal{A}x=b$, $\widetilde{\mathcal{A}}\widetilde{x}=\widetilde{b}$ the following estimate holds true
$$
\| x-\widetilde{x}\| \leq \frac{\|{\mathcal{A}}^{-1}\|}{1-q}\left(\|b\|\,\|{\mathcal{A}}^{-1}\| \,\| \mathcal{A}-\widetilde{\mathcal{A}}\| +
\| b-\widetilde{b}\|\right).
$$ 
\end{lemma}
\begin{lemma}\label{lem2} Let the element $\mathcal{A}$ of a Banach algebra is invertible
and the element $\widetilde{\mathcal{A}}$ satisfies the inequality \eqref{morestrict}. 
Then $\widetilde{\mathcal{A}}$ is invertible element
and
$$
\|{\mathcal{A}}^{-1}-\widetilde{\mathcal{A}}^{-1} \| 
\leq \frac{1}{1-q}\, \|\mathcal{A}^{-1}\|^2\cdot \|{\mathcal{A}}-\widetilde{\mathcal{A}} \|.
$$
\end{lemma}

\begin{theorem}\label{invplus}
Let $A(t)\in GW^{p\times p}$ admits the canonical factorisation
$
A(t)=A_{-}(t)A_{+}(t)
$
and this factorisation is normalised by the condition $A_{-}(\infty)=A_{0}$.
If the matrix function $\widetilde{A}(t)\in W^{p\times p}$ satisfies the inequality
\begin{equation}
\label{estim}
{\|A-\widetilde{A}\|}_{W} \le \frac{q}{{\|A_{+}^{-1}\|}_{W} {\|A_{-}^{-1}\|}_{W}}, \quad 0<q<1,
\end{equation}
then $\widetilde{A}(t)$ is an invertible matrix function that admits the canonical factorisation
$
\widetilde{A}(t)=\widetilde{A}_{-}(t) \widetilde{A}_{+}(t).
$
For the unique factorisation normalised by the same condition $\widetilde{A}_{-}(\infty)=A_{0}$, 
the following estimates are valid:
\begin{equation}\label{est1plus}
(i)\quad {\| A_{+}^{-1}-\widetilde{A}_{+}^{-1}\|}_{W}\leq \frac{\|A_{0}\|}{1-q}\,{\|A_{+}^{-1}\|}_{W}^{2}\, {\| A_{-}^{-1}\|}_{W}^{2} {\|A-\widetilde{A}\|}_{W},
\end{equation}
\begin{multline}\label{estAminustildeA}
(ii)\quad {\|A_{-}-\widetilde{A}_{-}\|}_{W} \leq\biggl(\|A_{+}^{-1}\|_{W}+
\frac{\|A_0\|\|A\|_W}{1-q}\|A_{+}^{-1}\|_{W}^{2} {\|A_{-}^{-1}\|}_{W}^{2}+\\
+\frac{q}{1-q}\|A_0\|\|A_{+}^{-1}\|_{W} {\|A_{-}^{-1}\|}_{W}\biggr)\|A-\widetilde{A}\|{_{W},}
\end{multline}
Moreover, if condition \eqref{estim} is replaced by a stronger one:
\begin{equation}\label{newradius}
{\|A-\widetilde{A}\|}_{W} \le \frac{q}{\|A_{+}^{-1}\|_{W} {\|A_{-}^{-1}\|}_{W}}\cdot\frac{(1-q)}{{\|A_0\|\|A_{-}^{-1}\|_{W}}\|A_+\|_W\|A^{-1}_+\|_W},
\end{equation}
then
\begin{equation}\label{AtAplus}
(iii)\quad
{\|A_{+}-\widetilde{A}_{+}\|}_{W} \leq \frac{\|A_{0}\|}{(1-q)^2}
 {\|A_{+}\|}_{W}^{2} {\|A_{+}^{-1}\|}_{W}^{2} {\|A_{-}^{-1}\|}_{W}^{2} 
\| {\|A-\widetilde{A}\|}_{W}.
\end{equation}
\end{theorem}

\begin{proof} 
To prove the first inequality (\ref{est1plus}), we consider operator equations $T_{A}X=A_{0}$ and $T_{\widetilde{A}} \widetilde{X}=A_{0}$, that both have unique solutions. These solutions are matrix functions $X=A_{+}^{-1}$ and $\widetilde{X}=\widetilde{A}_{+}^{-1}$, respectively. Indeed, $T_{A} A_{+}^{-1}=T_{A_{-}} T_{A_{+}} A_{+}^{-1}= \mathcal{P}_{+} A_{-}  \mathcal{P}_{+} A_{+} A_{+}^{-1}=\mathcal{P}_{+} A_{-} I_{p}=\mathcal{P}_{+} A_{-}=A_-(\infty)=A_{0}$. Similar arguments apply for the second equation.

Utilising Lemma \ref{lemxx} with $b=\widetilde b =A_0$, we have
$$
{\| A_{+}^{-1}-\widetilde{A}_{+}^{-1}\|}_{W}= {\| X-\widetilde{X}\|}_{W} 
\leq \frac{\|A_{0}\|}{1-q}\, {\|T_{A}^{-1}\|}^{2} \,\| T_{A}-T_{\widetilde{A}}\| .
$$
Finally, since $\|T_{A}^{-1}\| \le {\| A_{+}^{-1}\|}_{W} {\|A_{-}^{-1}\|}_{W}$ and 
$\| T_{A}-T_{\widetilde{A}}\|\leq {\|A-\widetilde{A}\|}_{W} $, we arrive at (\ref{est1plus}).

Now we can prove the continuity of the factor $A_{-}$ with its explicit estimate (ii).
Since both matrix functions $A(t)$ and $\tilde A(t)$ admit canonical factorisations, the following identity is valid:
\[
{{A}_{-}}-{{\widetilde{A}}_{-}}=(A-\widetilde{A})A_{+}^{-1}+\widetilde{A} (A_{+}^{-1}-\widetilde{A}_{+}^{-1}),
\]
leading to the inequality
\begin{equation}\label{intermid}
{\|A_{-}-\widetilde{A}_{-}}\|_{W} \leq {\|A_{+}^{-1}\|}_{W} \|A-\widetilde{A}\|+\|\widetilde{A}\|_{W} \|A_{+}^{-1}-\widetilde{A}_{+}^{-1}\|.
\end{equation}
Utilising (\ref{estim}), we get an estimate of the first multiplayer of the second term, $\|\widetilde{A}\|_{W} $, in the equality (\ref{intermid}) in the form
$$
\| \widetilde{A}\|_{W} \leq \|A\|_{W} +{\|A-\widetilde{A}\|}_{W} 
\leq  \|A\|_{W}+\frac{q}{ {\|A_{+}^{-1}\|}_{W} {\|A_{-}^{-1}\|}_{W}},
$$
while the statement (i) of this theorem provides an estimate for the second multiplayer.
Substituting those two inequalities into (\ref{intermid}) we arrive at \eqref{estAminustildeA}.

Finally, we can prove the statement (iii), that represents an explicit stability
condition for the factor $A_{+}(t)$ for a small perturbation of the matrix function $A(t)$.

Since condition (\ref{newradius}) is
stronger than \eqref{estim}, it also guarantees an existence of the canonical factorisation 
$\widetilde{A}(t)$ 
fixed by its limiting value $A_-(\infty)=A_0$. 
Utilising Lemma \ref{lem2} to elements
$\mathcal{A}=A_{+}^{-1}(t)$, 
$\widetilde{\mathcal{A}}=\widetilde{A}_{+}^{-1}(t)$ considered in the Banach algebra
$W_{+}^{p\times p}$, and taking into account the statement (i), that is the inequality \eqref{est1plus}  proven above, we arrive at  \eqref{AtAplus}. 
Condition (\ref{morestrict}) is  fulfilled if we take a new radius of the vicinity of $A(t)$.
 \end{proof}
%Then
%$$
%{\|A_{+}-\widetilde{A}_{+}\|}_{W} \leq \frac{\|A_{0}\|}{(1-q)^2}
% {\|A_{+}\|}_{W}^{2} {\|A_{+}^{-1}\|}_{W}^{2} {\|A_{-}^{-1}\|}_{W}^{2} 
%\| {\|A-\widetilde{A}\|}_{W}.
%$$

To finalise this subsection, we note that the explicit estimates for the absolute errors
$
\|A_{+}-\widetilde{A}_{+}\|_{W},\,  
 \| A_{-}-\widetilde{A}_{-}\|_{W}
$
were obtained only for canonical factorisation of the original matrix function $A(t)$ of an arbitrary size $p\times p$.

\subsection{An explicit estimates for factors $a_\pm(t)$ for matrix function $a(t)$ from (\ref{aform})} \label{estimapprox}
In this subsection, we specify the results for matrix function $a(t)$ in the form (\ref{aform}) for $\theta=2\nu$ and assume that the partial indices are equal. Then we can always consider a new matrix function $b(t)=t^{-\nu}a(t)$ admitting classic canonical factorisation ($\theta=0$). Below, we assume without lost of generality, that $\theta=0$ for both matrix functions $a(t)$ and 
$a_N(t)$. We will assume that the factorisation is normalized by the condition $a_-^{(N)}(\infty)=I_2$.
Applying Th.\ref{invplus} to the matrix function $a_N(t)$, we can formulate
\begin{corollary}\label{cor5.1}
Let for some $N$, $\zeta_1,\zeta_2$ we have
$
\|a-a_N\|_W<\delta_N(\zeta_1,\zeta_2),
$
$a_N(t)$ admits the canonical factorisation.
Then 
\begin{equation}
\label{5_1a}
  \|a_+^{-1}-(a_+^{(N)})^{-1}\|_W<\frac{\|(a_+^{(N)})^{-1}\|_W^2\,
\|(a_-^{(N)})^{-1}\|_W^2}{1-q_N(\zeta_1,\zeta_2)}\delta_N(\zeta_1,\zeta_2),
\end{equation}
if $q_N(\zeta_1,\zeta_2):=\delta_N(\zeta_1,\zeta_2)\|\bigl(a_+^{(N)}\bigr)^{-1}\|_W\,\|\bigl(a_-^{(N)}\bigr)^{-1}\|_W<1$;
\begin{multline}\label{quraccm}
 \|a_--a_-^{(N)}\|_W< \delta_-=\Bigl(\|(a_+^{(N)})^{-1}\|_W+
\frac{\|a_N\|_W\,\|(a_+^{(N)})^{-1}\|_W^2
\|(a_-^{(N)})^{-1}\|_W^2}{1-q_N(\zeta_1,\zeta_2)}+\\
 \frac{q_N(\zeta_1,\zeta_2)\,\|(a_+^{(N)})^{-1}\|_W\,
\|(a_-^{(N)})^{-1}\|_W}{1-q_N(\zeta_1,\zeta_2)}\Bigr)\delta_N(\zeta_1,\zeta_2),
\end{multline}
if $q_N(\zeta_1,\zeta_2)<1$;
\begin{equation}\label{quraccp}
 \|a_+-a_+^{(N)}\|_W<\delta_+=\frac{\|a_+^{(N)}\|_W^2\,\|(a_+^{(N)})^{-1}\|_W^2\,
\|(a_-^{(N)})^{-1}\|_W^2}{{\bigl(1-q^+_N(\zeta_1,\zeta_2)\bigr)^2}}\delta_N(\zeta_1,\zeta_2),
 \end{equation}
if $\gamma_N(\zeta_1,\zeta_2):=4\delta_N(\zeta_1,\zeta_2)
{\|a_+^{(N)}\|_W\|(a_+^{(N)})^{-1}\|_W^2\|(a_-^{(N)})^{-1}\|_W^2}\leq 1.$

Here $q^+_N(\zeta_1,\zeta_2)=\frac{1}{2}\bigl(1-\sqrt{1-\gamma_N(\zeta_1,\zeta_2)}\,\bigr).$
\end{corollary}
\begin{proof}
The proof is required only for estimate (\ref{quraccp}). Let us define the parameter $q^+_N(\zeta_1,\zeta_2)$ by the equation 
\begin{equation}\label{qpluseq}
q(1-q){=}\frac{1}{4}\gamma_N(\zeta_1,\zeta_2).
\end{equation}
Since $\max_{[0,1]}q(1-q)=\frac{1}{4}$, the condition $\gamma_N(\zeta_1,\zeta_2)\leq 1$ must be met. We have two roots of Eq.(\ref{qpluseq}), but the root is such that $0<q\leq 1/2$, that is $q^+_N(\zeta_1,\zeta_2)=\frac{1}{2}\bigl(1-\sqrt{1-\gamma_N(\zeta_1,\zeta_2)}\,\bigr)$, gives the better estimate for $\|a_+-a_+^{(N)}\|_W$. Now all conditions of Th.\ref{invplus}, item (iii), are fulfilled and applying it we arrive to (\ref{quraccp}).
\end{proof}

Note that, since the factorisation problem for a Laurent matrix polynomial $a_N(t)$ can be solved explicitly (see \cite{Adukov92}), the estimates for 
$\|a_--a_-^{(N)}\|_W$ and $\|a_+-a_+^{(N)}\|_W$ can be obtained constructively in terms of some characteristics of the functions 
$\alpha_+(t)$, $\beta_-(t)$.

Suppose that the Laurent coefficients of $\alpha_+^{(N)}(t)$, $\beta_-^{(N)}(t)$ belong  to the field $\mathbb Q(i)$. Since $\det a_N(t)=1$, the factorisation problem for $a_N(t)$ can be solved exactly \cite{AAM22}. We can construct the exact factorisation with help of the package {\sf ExactMPF} \cite{AAM22}. 
If the Laurent coefficients $\alpha_k$, $\beta_k$  do not belong the field $\mathbb Q(i)$, we can find the best rational approximations for them and include the respective corrections in $\delta_N(\zeta_1,\zeta_2)$.

\section{Numerical experiments}\label{exp}

\subsection{Matrix function admitting canonical factorisation ($\theta= 0$).} 
First we consider matrix function $a(t)$:
\begin{equation}
\label{case1}
a(t)=\begin{pmatrix}
1&t^{-2} \sqrt{ k_1^2-t^2}\\[2mm]
\sqrt{ k_2^2-t^2} &1+t^{-2} \sqrt{ k_2^2-t^2} \sqrt{ k_1^2-t^2}
\end{pmatrix},\quad k_1^{-1},k_2>1.
\end{equation}

Here, $\alpha_+(t)$ is the branch of $\sqrt{ k_2^2-t^2}$, $k_2>1$, that is analytic into the domain $\mathbb C \setminus \bigl((-\infty, -k_2)
\cup (k_2, +\infty)\bigr)$, and has
the positive value at $t=0$.  

For the function $\beta_-(t)=t^{-2}\sqrt{ k_1^2-t^2}$, $k_1<1$, we consider the branch of the square root $\sqrt{ k_1^2-t^2}$ with the cut $(-k_1, k_1)$ fixed by value of $t^{-1}\sqrt{ k_1^2-t^2}$ at $t=\infty$ equal to $i$.

The function $\alpha_+(t)$ is expanded into the series
$$
\alpha_+(t)=k_2\Bigl(1-\frac{t^2}{2k_2^2}-\sum_{n=2}^{\infty}
\frac{(2n-3)!!}{(2n)!!}\frac{t^{2n}}{k_2^{2n}}\Bigr),
$$
that converges absolutely  in the closet disk $|t|\leq k_2$.
So, $\alpha_+(t)$
is continuous on $|t|\leq k_2$, analytic into $|t|< k_2$, and
$\alpha_+(t)$, $|t|=1$, belongs the subalgebra $W_+$.
We can easily prove that $M_+(\zeta_2)=\sqrt{k_2^2+\zeta_2^2}$ for
$1<\zeta_2\leq k_2$, where function $M_+(\zeta_2)$ has been defined in \eqref{alpha_est}. 

Function $\beta_-(t)$ is expanded into the series
$$
\beta_-(t)=\frac{i}{t}\Bigl(1-\frac{k_1^2}{2t^2}-\sum_{n=2}^{\infty}
\frac{(2n-3)!!}{(2n)!!}\frac{k_1^{2n}}{t^{2n}}\Bigr),
$$
that converges absolutely  in the close domain $|t|\geq k_1$.
Moreover, $\beta_-(t)$ is continuous on $|t|\geq k_1$, analytic into 
$|t|> k_1$, and $\beta_-(t)$, $|t|=1$, belongs the subalgebra $W_-^0$.
It is  easy to prove that 
$M_-(\zeta_1)=\frac{\sqrt{k_1^2+\zeta_1^2}}{\zeta_1^2}$ for $k_1\leq \zeta_1<1$.
Thus, the admissible pairs $(\zeta_1,\zeta_2)$ belong to the rectangular
$[k_1,1)\times (1, k_2]$.

Now we have $\|a-a_N\|_W<\delta_N(\zeta_1,\zeta_2)$ for $k_1\leq\zeta_1<1$, $1<\zeta_2\leq k_2$, with the parameter $\delta_N(\zeta_1,\zeta_2)$ defined in
\eqref{delta_N12}
\begin{multline}\label{estdelta}
\delta_N(\zeta_1,\zeta_2)=\\\Biggl(1+\frac{\zeta_2\sqrt{k_2^2+\zeta_2^2}}{\zeta_2-1}\Biggr)\frac{\zeta_1^{N-1}\sqrt{k_1^2+\zeta_1^2}}{1-\zeta_1}+\Biggl(1+\frac{\sqrt{k_1^2+\zeta_1^2}}{\zeta_1(1-\zeta_1)}\Biggr)\frac{\sqrt{k_2^2+\zeta_2^2}}{\zeta_2^N(\zeta_2-1)}.
\end{multline}

Since the optimal choice of the parameters $\zeta_1,\zeta_2$  is desirable but not mandatory, we choose as optimal values 
$\zeta_1,\zeta_2$ the solution of the minimization problem on a compact set $[k_1,1-\epsilon]\times [1+\epsilon, k_2]$ for sufficiently small $\epsilon>0$. For these purposes we use the package {\sf Optimization} of the system Maple.

\begin{example}\label{ex1A}
Now we present an approximate factorisation of matrix function \eqref{case1} with a guaranteed accuracy 
for a fixed values of $k_1=1/5$, $k_2=5$. With help procedure {\sf Minimize} of the  package {\sf Optimization} we verify that, for $\epsilon=0.01$ and for all $N\in [1, 100]$, optimal values are $\zeta_1=1/5$, $\zeta_2=5$, thus  
\begin{equation}
\label{delta_6_1}
\delta_N(1/5,5)=\frac{15+2\sqrt{2}}{4}\cdot 5^{1-N}\approx 22.2855\cdot 5^{-N}. 
\end{equation}

For any fixed $N_0$ and  $1\leq N\leq N_0$, the evident inequality 
\begin{equation}
\label{est6_1}
\|a_{N_0}-a_N\|_W\leq \|a-a_N\|_W\leq \delta_N.
\end{equation}
We found the norms 
$\|a_{30}-a_N\|_W$, $\|a_{40}-a_N\|_W$, $\|a_{50}-a_N\|_W$ and compared them with the calculated values of $\delta_N$ 
%for $N\in[1,20]$. The results are shown in Figure \ref{5_1_a} and Table \ref{accdeltaN1}. All calculations were carry out with 
for $N\in[1,30]$. The results are shown in Figure \ref{5_1_a} and Table 1. All calculations were carried out with 
$\rm Digits=10$. The values for $\|a_{50}-a_N\|_W$ completely coincide with $\|a_{40}-a_N\|_W$. As a result, we conclude that $\|a_{40}-a_N\|_W=\|a-a_N\|_W$ for $1\leq N \leq 30$. Note that the inequality $\|a-a_N\|_W<\delta_N$ is far from being exact: its left-hand side is estimated as $6^{-N}$  (see Figure~\ref{5_1_a}), while the right-hand side is of the order $22.3\cdot 5^{-N}$. 
From the first glance, since we know representations of the functions $a(t)$ and $a_N(t)$, we can compute also the norm $\|a-a_N\|_W$ itself:
\[
\|a-a_N\|_W = \sum\limits_{k=N}^{\infty }{||a_{k}||}. 
\]
Unfortunately, excluding the case when the series computed exactly, we can have only an estimate of the norm from below, while the values of $\delta_N$ even is being sometimes an overestimation, guaranties the inequality from the other side. 
Moreover, the value of $\delta_N$ is also involved in the computations of the second condition $q_N<1$ that is equally important for the Theorems~\ref{basethm} and \ref{basethm1} to be valid. 

\begin{figure}[h]
\centering
\includegraphics[width=0.7\linewidth]{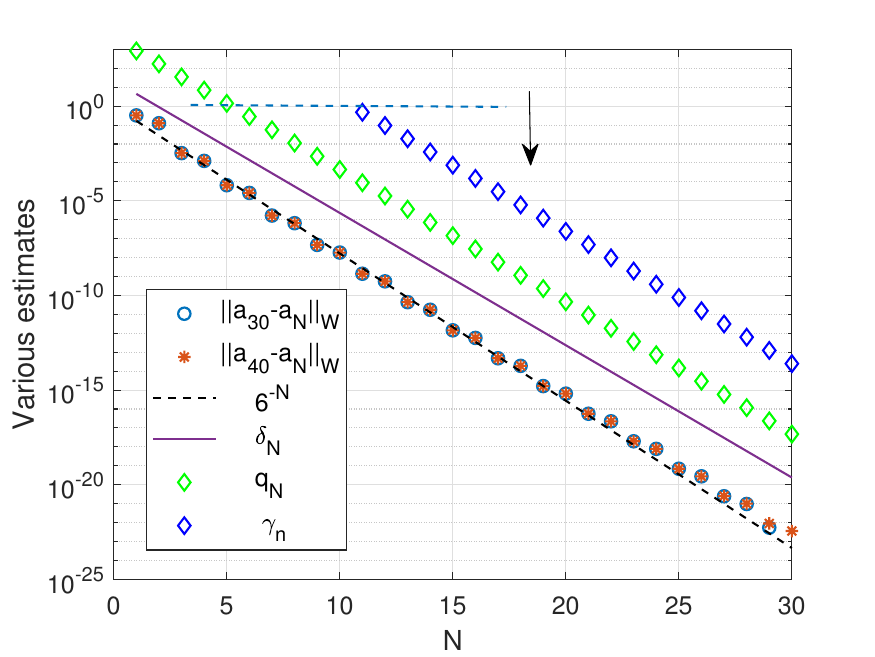}
 %\caption{Table 1}
 \caption{Verification of convergence of the approximation and estimate (\ref{estdelta}) for Example \ref{ex1A}. }
	\label{5_1_a}
\end{figure}

Note that $\delta_N$ is also involved in the estimates of the factors \eqref{5_1a}, \eqref{quraccm}. %\eqref{quraccp}.
Since the Laurent coefficients of $\alpha_+^{(N)}(t)$, $\beta_-^{(N)}(t)$ belong
the field $\mathbb Q(i)$ and $\det a_N(t)=1$, the factorisation problem
for $a_N(t)$ can be solved exactly. We apply the package {\sf ExactMPF} for solving the factorisation problems for $a_N(t)$, $1\leq N\leq 30$. 

It turns out that all $a_N(t)$ admit the canonical factorisations, which we normalised by the conditions 
%$a_-^{(N)}(\infty)=I_2$. In Table \ref{crit1}, we also show the 
$a_-^{(N)}(\infty)=I_2$. In Table 1, we also show the 
%$a_-^{(N)}(\infty)=I_2$. In Table \ref{accdeltaN1}, we also show the 
values of $\|\bigl(a_\pm^{(N)}\bigr)^{-1}\|_W$ and $q_N$. For clarity, before to be included into the tables, all exact rational numbers are converted into the floating format preserving their meaningful length.

Starting from $N=6$, the conditions of Theorem \ref{basethm} are fulfilled. Therefore, the matrix functions $a(t)$ and its approximation, $a_N(t)$, both admit canonical factorisation.

Since the formulas for factors are cumbersome, we present them in supplementary material.

In Table 2, we also present the values  of the right-hand sides of the inequalities (\ref{quraccm}), (\ref{quraccp})
denoted by $\delta_-^{(N)}$, $\delta_+^{(N)}$, respectively. Note that the equality (\ref{quraccp}) is valid if $\gamma_N \leq 1$ (see Corollary \ref{cor5.1}). From Table 1 we see that this condition is fulfilled from $N=11$. Hence we calculate $\delta_{+}^{(N)}$ beginning from $N=11$.
%In Table~\ref{crit1}, we also present the values  of the right-hand sides of the inequalities (\ref{quraccm}), (\ref{quraccp})
%denoted by $\delta_-^{(N)}$, $\delta_+^{(N)}$, respectively.  
They are estimations of the respective norms  $\|a_\pm^{(N)}-a_\pm\|_W$. 
For example, $\|a_+-a_+^{(15)}\|_W<1.275\cdot \,10^{-3}$ and 
$\|a_--a_-^{(15)}\|_W<3.490\cdot \,10^{-4}$. In fact, the actual accuracy is much higher.
Utilising the same line of reasoning as for the matrix functions above, we present in Figure~\ref{5_1_b} 
the norms $\|a_\pm^{(40)}-a_\pm^{(N)}\|_W$ and their estimates $\delta_\pm^{(N)}$ for $1\leq N\leq 30$. 
Again, theoretical estimate for the factors are rather crude in this case. 

The respective data are given in Table 2 in Supplementary material, where we also present the norms $\|a_\pm^{(N)}-a_\pm^{(N-1)}\|_W$. 

\begin{figure}[h!]
\centering
\includegraphics[width=0.7\linewidth]{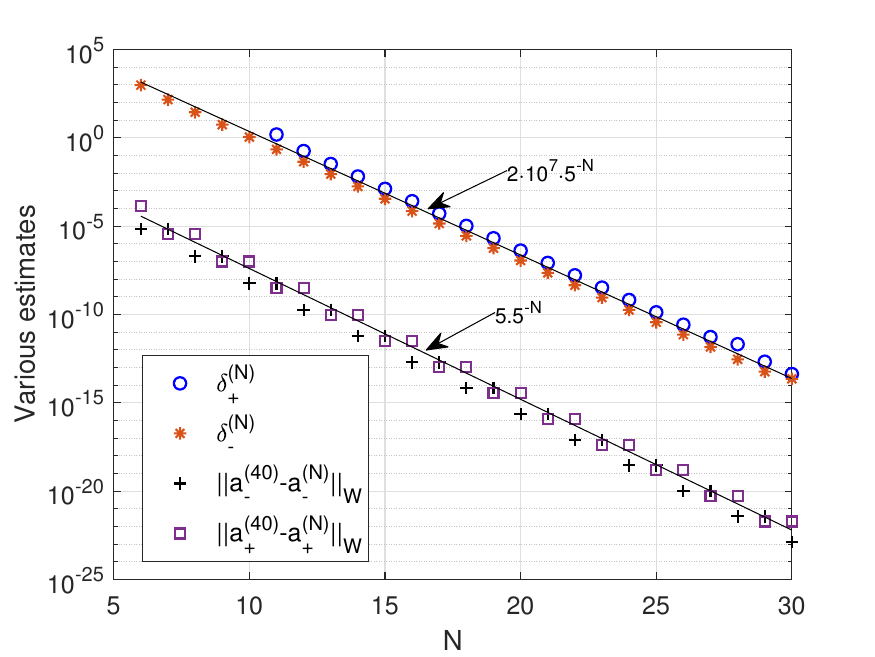}
  \caption{Estimates $\delta_\pm^{(N)}$ and 
the norms $\|a_\pm^{(40)}-a_\pm^{(N)}\|_W$ for Example \ref{ex1A}. }
 	\label{5_1_b}
\end{figure}
\end{example}

\subsection{Matrix function with equal partial indices, 
($\theta=2\nu$).}

Now we consider other example.
$$
a(t)=\begin{pmatrix}
1&\frac{1}{\sqrt{ k_1^2+t^2}}\\[3mm]
e^{k_2t}&t^{2\nu}+\frac{e^{k_2t}}{\sqrt{ k_1^2+t^2}}
\end{pmatrix},\quad 0<k_1<1,\quad k_2\in\mathbb C.
$$
Here $\alpha_+(t)=e^{k_2t}$, $k_2\in\mathbb C$. It is an entire function, represented by its Taylor series 
$$\alpha_+(t)=\sum_{n=0}^{\infty}\frac{k_2^nt^n}{n!},
$$ 
and the function $M_+$ defined in \eqref{alpha_est} is $M_+(\zeta_2)=e^{|k_2|\zeta_2}$, $1<\zeta_2<\infty$.

The other function, $\beta_-(t)=(k_1^2+t^2)^{-1/2}$, $0<k_1<1$, with the branch cut $(-k_1i, k_1 i)$ on the imaginary axis, while the branch is fixed by the limiting value $t\beta_-(t)\to1$ as $t=\infty$ is equal to $1$. The function $\beta_-(t)$ is analytic into 
$|t|> k_1$ and  $\beta_-(t)$ is expanded into the series
$$
\beta_-(t)=\frac{1}{t}\Bigl(1+\sum_{n=1}^{\infty}
(-1)^{n}\frac{(2n-1)!!}{(2n)!!}\frac{k_1^{2n}}{t^{2n}}\Bigr),
$$
that converges absolutely  in the open domain $|t|> k_1$.
Moreover, $\beta_-(t)$, for $|t|=1$, belongs the subalgebra $W_-^0$.
It is  easy to prove that 
$M_-(\zeta_1)=(\zeta_1^2-k_1^2)^{-1/2}$ for $k_1< \zeta_1<1$.
From (\ref{estaaN2}) we have for $k_1<\zeta_1<1$, $1<\zeta_2<\infty$:
\begin{multline*}
\|a-a_{N}\|_W\leq\delta_N(\zeta_1,\zeta_2)=\\
\Biggl(1+\frac{\zeta_2e^{|k_2|\zeta_2}}{\zeta_2-1}\Biggr)\frac{\zeta_1^{N+1}}{(1-\zeta_1)\sqrt{\zeta_1^2-k_1^2}}+
\Biggl(1+\frac{\zeta_1}{(1-\zeta_1)\sqrt{\zeta_1^2-k_1^2}}\Biggr)
\frac{e^{|k_2|\zeta_2}}{\zeta_2^N(\zeta_2-1)}.
\end{multline*}

\begin{example}\label{ex2}
Let $k_1=1/5$, $k_2=1$, $\theta=2\nu=6$. Choosing an optimal choice of the parameters $\zeta_1$, $\zeta_2$, we observed that the pair has small effect on the factorisation process. Thus, we took them basing on the convenience of the numerical calculations. 

Numerical experiments lead us to the following values:
$\zeta_1=1/4$, $\zeta_2=4$, and consequently:
\begin{equation}
\label{delta_N_6_2}
\delta_N = \frac{109 e^4+60}{27}\frac{1}{\,4^N}\approx 222.6364 \cdot 4^{-N}.
\end{equation}
Similarly to the Example \ref{ex1A}, we compute the norms 
$\|a_{30}-a_N\|_W$, $\|a_{40}-a_N\|_W$, $\|a_{50}-a_N\|_W$ and compared them with the calculated values of $\delta_N$ for $N\in[1,30]$. The values for $\|a_{50}-a_N\|_W$ completely coincide with $\|a_{40}-a_N\|_W$ with the chosen accuracy (DIGID=10). 
Thus we can assume that the lower bound $\|a_{40}-a_N\|_W\leq \|a-a_N\|_W$ is exact and its left hand side fully characterizes the accuracy of the estimate
(\ref{estdelta}). The results are shown in Figure \ref{5_2_a} and in the Table 3.

\begin{figure}[h!]
\centering
\includegraphics[width=0.7\linewidth]{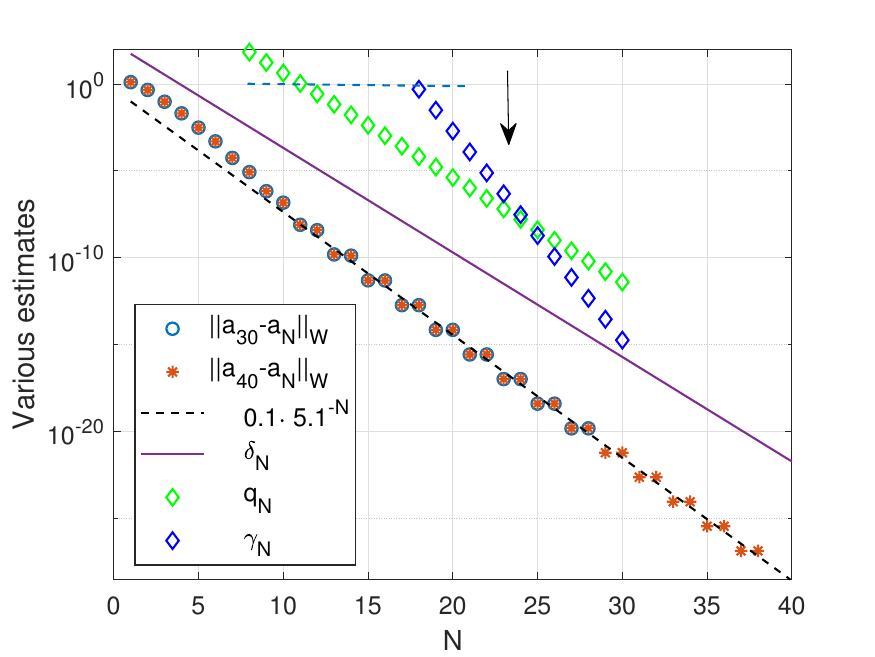}

 \caption{Verification of accuracy of the estimate (\ref{estdelta}) for Example \ref{ex2}.}
 	\label{5_2_a}
\end{figure}

\begin{figure}[h!]
\centering
\includegraphics[width=0.7\linewidth]{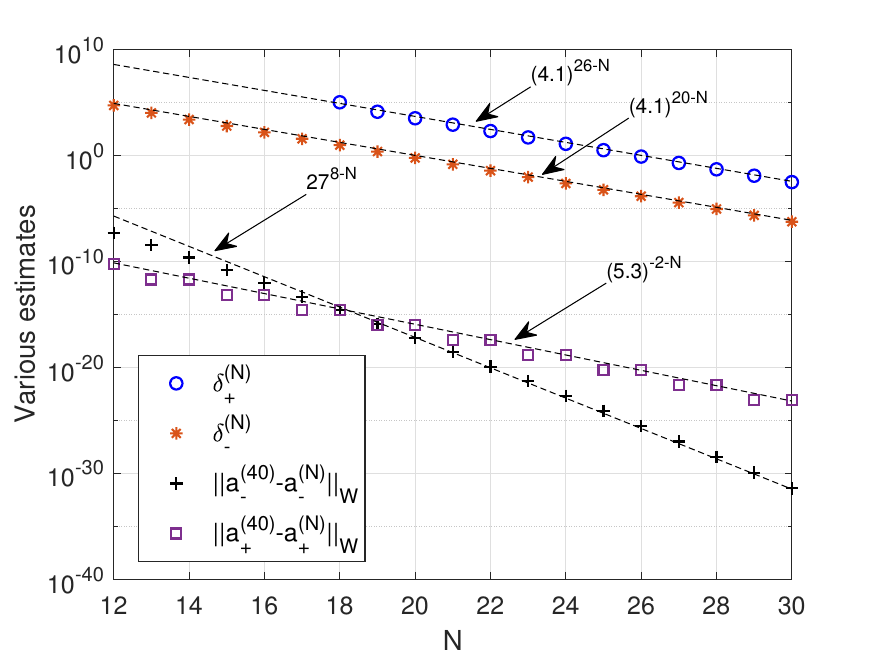}
 \caption{Comparison between $\delta_\pm^{(N)}$ with 
$\|a_\pm^{(40)}-a_\pm^{(N)}\|_W$ for Example \ref{ex2}. }
 	\label{5_2_b}
\end{figure}

As above, the Laurent coefficients of $\alpha_+^{(N)}(t)$, 
$\beta_-^{(N)}(t)$ belong
the field $\mathbb Q(i)$ and $\det a_N(t)=1$, hence the factorisation problem for $a_N(t)$ can be solved exactly with the package 
{\sf ExactMPF}. 

It turns out that all $a_N(t)$, $3\leq N\leq 30$, admit the stable factorisations with the equal indices $\rho_1=\rho_2=\nu=3$, in other words,
$a_N(t)=t^3a_-^{(N)}(t)a_+^{(N)}(t)$. We normalised the factorisations by the conditions $a_-^{(N)}(\infty)=I_2$. Obviously that estimates (\ref{quraccm}), (\ref{quraccp}) hold true in this case.

% In Table \ref{crit2} we show the 
%In Table \ref{accdeltaEx2} we show the 
In Table 3 we show the 
values of $\|\bigl(a_\pm^{(N)}\bigr)^{-1}\|_W$ and $q_N(1/4,4)$ (see \eqref{q_n}).
From the table we see that starting from $N=12$, the stability criterion is fulfilled. Therefore, the matrix function $a(t)$ admits the stable factorisation with equal indices $\rho_1=\rho_2=\nu=3$ and the stable factorisation of $a_N(t)$ gives an approximate stable factorisation of $a(t)$.

%NEW variant 
The calculations of the guaranteed accuracy give more or less satisfactory results for $N=30$: 
$\|a_+-a_+^{(30)}\|_W<3.012\cdot \,10^{-3}$ and 
$\|a_--a_-^{(30)}\|_W<5.474\cdot \,10^{-7}$.
The values of the norms 
$\|a_+^{(40)}-a_+^{(30)}\|_W=3.779\cdot \,10^{-32}, 
\|a_-^{(40)}-a_-^{(30)}\|_W= 8.741\cdot \,10^{-24}
$
show that the actual accuracy is much higher then predicted by the estimate.
Comparison $\delta_\pm^{(N)}$ with these norms are presented in Figure \ref{5_2_b} for all values of $N$ and in the Table 4.
%Comparison $\delta_\pm^{(N)}$ with these norms are presented in Figure \ref{5_2_b} for all values of $N$ and in the Table \ref{compEx2}.

\end{example}

\subsection{The stable factorisation of $a(t)$, $\theta=2\nu+1$.}

Now we consider the following matrix function 
$$
a(t)=\begin{pmatrix}
1&\frac{e^{k_1/t}}{t}\\[2mm]
e^{k_2t}& t^\theta+\frac{e^{k_1/t+k_2t}}{t}
\end{pmatrix},\quad k_1, k_2\in \mathbb C,\quad  \theta=2\nu+1.
$$
Here the function
$\beta_-(t)=t^{-1}e^{k_1/t}$ is analytic on $\mathbb C \setminus \{0\}$, and $\beta_-(\infty)=0$, while
$\alpha_+(t)=e^{k_2t}$ is an entire function. It is easy to check that
$M_-(\zeta_1) = \frac{1}{\zeta_1}e^{|k_1|/\zeta_1}$ and 
$M_+(\zeta_2)=e^{|k_2|\zeta_2}$, and (\ref{estaaN2}) gives for $0<\zeta_1<1$, $1<\zeta_2<\infty$:
$$
\delta_N(\zeta_1,\zeta_2)=
\Bigl(1+\frac{\zeta_2e^{|k_2|\zeta_2}}{\zeta_2-1}\Bigr)\frac{\zeta_1^{N}
e^{|k_1|/\zeta_1}}{1-\zeta_1}+\Bigl(1+\frac{e^{|k_1|/\zeta_1}}{1-\zeta_1}\Bigr)\frac{e^{|k_2|\zeta_2}}{\zeta_2^N(\zeta_2-1)},
$$

In the case of $\theta=2\nu+1$, the stability of the indices occurs at $\rho_1=\nu$, $\rho_2=\nu+1$.  
Unfortunately, estimates of the accuracy of the calculation of $\|a_\pm-a_\pm^{(N)}\|_W$ are not available, since the issue of the stability of P-normalisation has not been completed yet.

\begin{example}\label{ex3}
Let $k_1=1$, $k_2=1/2$, $\theta=-7$, i.e. $\nu=-4$.
If we put $\zeta_1=1/10$, $\zeta_2=10$, then
$$
\delta_N=\Bigl(
\frac{10}{9}e^{10}+\frac{110}{81}e^{15}+\frac{1}{9}e^{5}\Bigr)\frac{1}{10^N}\approx 
4.4639\cdot10^{6-N}.
$$

These values of the parameters $\zeta_1$, $\zeta_2$ give sufficiently
small values of $\delta_N$ beginning $N=13$ (see Table 5). 
%small values of $\delta_N$ beginning $N=13$ (see Table \ref{crit3norm}). 

Calculations with the package {\sf ExactMPF} show that for 
$3\leq N\leq 30$ the matrix functions $a_N(t)$ admit the stable factorisations with the indices $\rho_1=-4$, $\rho_2=-3$. The factorisation constructing by {\sf ExactMPF} admit $I_2$-normalisations
starting from $N=4$. We carry out these normalisations by formula 
(\ref{Inorm}). Respective estimate in graphical form are presented in Figure~\ref{5_3_a}. Table 5 with the data related to this example is given in the Supplementary Materials.

Then we calculated $\|\bigl(a_\pm^{(N)}\bigr)^{-1}\|_W$, and the parameter $q_N$ (see Table 5). The criterion $q_N<1$ is fulfilled beginning $N=24$. Hence the matrix function $a(t)$ admits the stable factorisation with the indices $\rho_1=-4$, $\rho_2=-3$.

\begin{figure}[h!]
\centering
\includegraphics[width=0.7\linewidth]{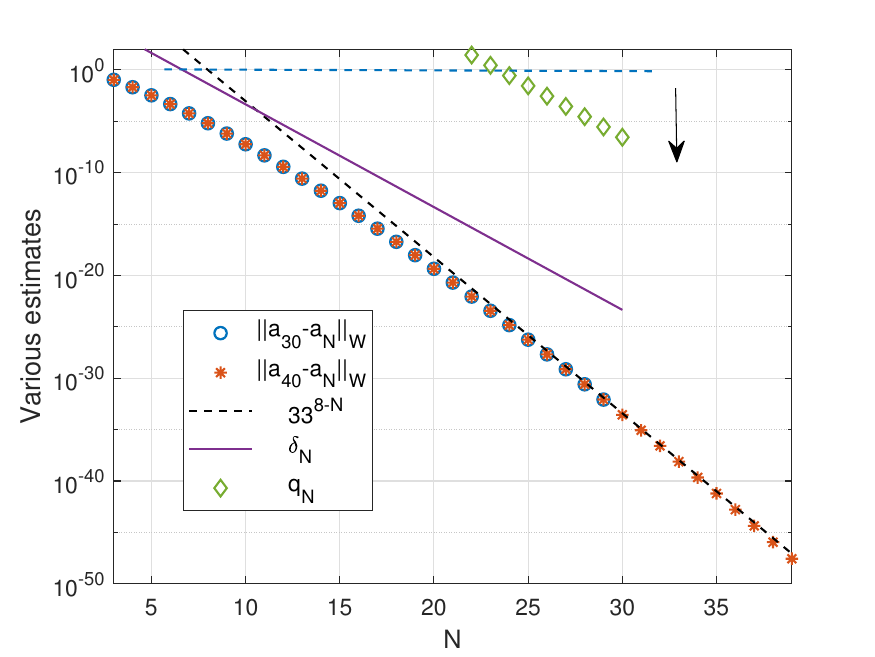}

\caption{Accuracy of the estimate (\ref{estdelta}) for Example \ref{ex3}. }
 	\label{5_3_a}
\end{figure}
\end{example}

\section{Conclusions}
We have developed a method allowing to confidently identify whether a strictly nonsingular  $2\times 2$  matrix function possesses stable set of partial indices simultaneously perform an approximate factorisation. For this purpose, we use 
the {\sf ExactMPF} package working within the Maple Software environment \cite{Maple} allowing for exact factorisation of an arbitrary nonsingular polynomial matrix function \cite{AAM22}. We have developed the efficient condition 
determining the stability region. 

Having this crucial information, we have constructed a sequence of polynomial matrix functions approximating the given matrix function in such a way that the distance between the initial matrix function and a consequent member of the approximation sequence decreases, being simultaneously smaller than the stability radius of the respective polynomial matrix function. This, in turn, proves the convergence of the sequence and the stability of the developed factorisation. 

The developed approach has been complimented with three different numerical examples demonstrating the method efficiency. We would like to underline, however, that theoretical estimate for the factor approximation in the case where the partial indices are stable but different (Example 6.3) is still to be developed since there are no effective estimates for the radius of neighbourhood where $A$ and $\tilde A$ admit the same $P$-normalisation.

{\bf Acknowledgements.}
NA and GM All gratefully acknowledge the support of the EU H2020 grant MSCA RISE-
2020-101008140-EffectFact. GM likes to thank the
Isaac Newton Institute for Mathematical Sciences (INI) for their support and
hospitality during the programme \lq\lq{}Mathematical theory and applications of
multiple wave scattering" (MWS), where on this paper was undertaken
and supported by EPSRC grant no. EP/R014604/1. GM
is also grateful for the funding received from the Simons Foundation that supported
their visit to INI during January-June 2023 and participation in MWS
programme.

\newpage
\pagenumbering{arabic}
\begin{center}
\section*{Supplementary material}
\setcounter{table}{0}
\renewcommand{\thetable}{\arabic{table}S}
for the article\\
A effective criterion for a stable factorisation of a strictly nonsingular $2\times2$ matrix functions. Utilisation of the 
{\sf ExactMPF} package
 \\
 N.\,V. Adukova, V.\,M. Adukov,  G. Mishuris
\end{center}

\noindent
{\bf Example 6.1}

\begin{table}[H]
{\tiny
\begin{center}
\begin{tabular}{|c|c|c|c|c|c|c|c|}
\hline
$N$&$\delta_N$&$\|(a_+^{(N)})^{-1}\|$&$\|(a_-^{(N)})^{-1}\|$&$q_N$&$\gamma_N$&$\|a_{30}-a_N\|_W$&$\|a_{40}-a_N\|_W$\\
\hline
1&4.457&6.000000000&31.00000000&829.021&$ 3.701 \,10^{6}$ &$3.252250175\,10^{-1}$&$3.252250175\,10^{-1}$\\
\hline
2&$8.914\,10^{-1}$&6.200000000&31.00000000&171.331 &$8.825 \,10^{5}$  &$1.252650255\,10^{-1}$&$1.252650255\,10^{-1}$\\
\hline
3&$1.783\,10^{-1}$&6.201000600&31.68815767&35.032 & $1.845 \,10^{5} $&$3.305033488\,10^{-3}$&$3.305033488\,10^{-3}$\\
\hline
4&$3.565\,10^{-2}$&6.213066452&31.68814324&7.020&$3.709 \,10^{4}$  &$1.285833689\,10^{-3}$&$1.285833689\,10^{-3}$\\
\hline
5&$7.131\,10^{-3}$&6.213066651&31.69452006&1.404&$7.420 \,10^{3}$  &$6.643388870\,10^{-5}$&$6.643388870\,10^{-5}$\\
\hline
6& $1.426\,10^{-3}$&6.213299946&31.69452005&$2.808\,10^{-1}$&$ 1.484 \,10^{3}$ &$2.604989431\,10^{-5}$&$2.604989431\,10^{-5}$\\
\hline
7&$2.852\,10^{-4}$&6.213299946&31.69464756&$5.617\,10^{-2}$& $2.968 \,10^{2}$ &$1.665819909\,10^{-6}$&$1.665819909\,10^{-6}$\\
\hline
8&$5.705\,10^{-5}$&6.213305778&31.69464756&$1.123\,10^{-2}$& $5.937 \,10^{1}$ &$6.562780772\,10^{-7}$&$6.562780772\,10^{-7}$\\
\hline
9&$1.141\,10^{-5}$&6.213305778&31.69465074&$2.247\,10^{-3}$& $1.187 \,10^{1}$ &$4.673619535\,10^{-8}$&$4.673619535\,10^{-8}$\\
\hline
10&$2.282\,10^{-6}$&6.213305942&31.69465074&$4.494\,10^{-4}$&2.375 &$1.847014324\,10^{-8}$&$1.847014324\,10^{-8}$\\
\hline
11&$4.564\,10^{-7}$&6.213305942&31.69465083&$8.988\,10^{-5}$& $4.750\,10^{-1}$&$1.404091080\,10^{-9}$&$1.404091080\,10^{-9}$\\
\hline
12&$9.128\,10^{-8}$&6.213305947&31.69465083&$1.797\,10^{-5}$&$9.499 \,10^{-2}$&$5.561335285\,10^{-10}$&$5.561335285\,10^{-10}$\\
\hline
13&$1.826\,10^{-8}$&6.213305947&31.69465084&$3.595\,10^{-6}$& $1.890\,10^{-2}$ &$4.417597702\,10^{-11}$&$4.417597702\,10^{-11}$\\
\hline
14&$3.651\,10^{-9}$&6.213305947&31.69465084&$7.190\,10^{-7}$&$3.799\,10^{-3}$  &$1.752644859\,10^{-11}$&$1.752644859\,10^{-11}$\\
\hline
15&$7.302\,10^{-10}$&6.213305947&31.69465084&$1.438\,10^{-7}$&$7.599\,10^{-4}$  &$1.436920153\,10^{-12}$&$1.436920153\,10^{-12}$\\
\hline
16&$1.460\,10^{-10}$&6.213305947&31.69465084&$2.876\,10^{-8}$&$1.520\,10^{-4}$ &$5.708247876\,10^{-13}$&$5.708247877\,10^{-13}$\\
\hline
17&$2.921\,10^{-11}$&6.213305947&31.69465084&$5.752\,10^{-9}$&$3.040\,10^{-5}$ &$4.792942211\,10^{-14}$&$4.792942215\,10^{-14}$\\
\hline
18&$5.842\,10^{-12}$&6.213305947&31.69465084&$1.150\,10^{-9}$ &$6.079\,10^{-6}$  &$1.905995816\,10^{-14}$&$1.905995820\,10^{-14}$\\
\hline
19&$1.168\,10^{-12}$&6.213305947&31.69465084&$2.301\,10^{-10}$ &$1.216\,10^{-6}$  &$1.630494217\,10^{-15}$&$1.630494251\,10^{-15}$\\
\hline
20&$2.337\,10^{-13}$&6.213305947&31.69465084&$4.602\,10^{-11}$&$2.432\,10^{-7}$ &$6.489430572\,10^{-16}$&$6.489430919\,10^{-16}$\\
\hline
21&$4.674\,10^{-14}$ &6.213305947 &31.69465084 &$9.204\,10^{-12}$ &$4.863\,10^{-8}$  & $5.635189764\,10^{-17}$&$5.635193238\,10^{-17}$\\
\hline
22&$9.347\,10^{-15}$ &6.213305947 &31.69465084 &$1.841\,10^{-12}$ &$9.723\,10^{-9}$  & $2.244407226\,10^{-17}$&$2.244410699\,10^{-17}$\\
\hline
23&$1.869\,10^{-15}$ &6.213305947 &31.69465084 &$3.681\,10^{-13}$ & $1.945\,10^{-9}$ &$1.973046872\,10^{-18}$ &$1.973081604\,10^{-18}$\\
\hline
24&$3.739\,10^{-16}$ &6.213305947 &31.69465084 &$7.363\,10^{-14}$ &$3.891\,10^{-10}$  & $7.862820339\,10^{-19}$&$7.863167659\,10^{-19}$\\
\hline
25&$7.478\,10^{-17}$ &6.213305947 &31.69465084 &$1.473\,10^{-14}$ &$7.782\,10^{-11}$  & $6.980519510\,10^{-20}$&$6.983992739\,10^{-20}$\\
\hline
26&$1.495\,10^{-17}$&6.213305947 &31.69465084 &$2.945\,10^{-15}$ &$1.556\,10^{-11}$  &$2.781225221\,10^{-20}$ &$2.784698451\,10^{-20}$\\
\hline
27&$2.991\,10^{-18}$ &6.213305947 &31.69465084 &$5.890\,10^{-16}$& $3.113\,10^{-12}$ & $2.460258627\,10^{-21}$&$2.495001635\,10^{-21}$\\
\hline
28&$5.982\,10^{-19}$ &6.213305947 &31.69465084 &$1.1783\,10^{-16}$ &$6.225\,10^{-13}$  &$9.605187746\,10^{-22}$ &$9.952621710\,10^{-22}$\\
\hline
29&$1.196\,10^{-19}$ &6.213305947 &31.69465084 &$2.356\,10^{-17}$ &$1.245\,10^{-13}$  &$5.402914661\,10^{-23}$ &$8.984270703\,10^{-23}$\\
\hline
30&$2.393\,10^{-20}$ &6.213305947 &31.69465084 &$4.712\,10^{-18}$ &$2.490\,10^{-14}$  &0 &$3.585235396\,10^{-23}$\\
\hline
\end{tabular}
\end{center}
}
\caption{Verification of the criterion $q_N<1$ and verification of accuracy of the estimate (\ref{estdelta}) for Example \ref{ex1A}.}\label{accdeltaN1}
\end{table}

For example, we have $a(t)\approx a_{-}^{(15)}(t) a_{+}^{(15)}(t)$. 
All elements of the factor $a_-^{(15)}(t)$:
\begin{multline*}
{\textstyle
\bigl(a_-^{(15)}(t)\bigr)_{11}=1.-\frac{5.003000600 i}{t}-\frac{0.1000500140\,10^{-1}}{t^2}+\frac{0.1000000100 i}{t^3}+}\\
{\textstyle -\frac{0.1000600220\,10^{-4}}{t^4}+\frac{0.9998000600\, 10^{-3} i}{t^5}-\frac{0.2001300540\,10^{-5}}{t^6}+\frac{0.1999400060\,10^{-4} i}{t^7}-}\\
{\textstyle \frac{5.003401521\,10^{-8}}{t^8}+\frac{4.998200040\,
10^{-7}i}{t^9}-\frac{1.400980460\,10^{-9}}{t^{10}}+
\frac{1.399439988\,10^{-8}i}{t^{11}}-}\\
{\textstyle \frac{4.203000600\,10^{-11}}{t^{12}}+
\frac{4.198200768\,10^{-10} i}{t^{13}}-\frac{1.320528106\,10^{-12}}{t^{14}} +\frac{1.320264158\,10^{-11}i}{t^{15}}}.
\end{multline*}
\begin{multline*}
{\textstyle
\bigl(a_-^{(15)}(t)\bigr)_{12}=\frac{1.000200000 i}{t}+
\frac{0.2000199880\,10^{-2}}{t^2}-\frac{0.2000400120\,10^{-1} i}{t^3}+}\\
{\textstyle
\frac{0.2000399960\,10^{-4}}{t^4}-
\frac{0.2000400160\,10^{-3} i}{t^5}+\frac{4.001000040\,10^{-7}}{t^6}-
\frac{0.4000800360\,10^{-5} i}{t^7}+}\\
{\textstyle
\frac{1.000280032\,10^{-8}}{t^8}-\frac{1.000200096\,10^{-7} i}{t^9}+\frac{2.800840136\,10^{-10}}{t^{10}}-
\frac{2.800560280\,10^{-9} i}{t^{11}}+}
\\
{\textstyle \frac{8.402638799\,10^{-12}}{t^{12}}-
\frac{8.401680864\,10^{-11} i}{t^{13}}+\frac{2.639999789\,10^{-13}}
{t^{14}}-\frac{2.640528106\,10^{-12} i}{t^{15}}}
\end{multline*}
\begin{multline*}
{\textstyle
\bigl(a_-^{(15)}(t)\bigr)_{21}=
- \frac{25.02500400 i}{t}-\frac{0.5001499900\,10^{-1}}{t^2}+\frac{0.4999000500 i}{t^3}- }\\
{\textstyle
\frac{0.5000999299\,10^{-3}}{t^4}+
\frac{0.4997000400\,10^{-2} i}{t^5}-\frac{0.1000149790\,10^{-4}}{t^6}+\frac{0.9992000700\,10^{-4}}{t^7}-}\\
{\textstyle
\frac{2.500299359\,10^{-7}}{t^8}+
\frac{0.2497700160\,10^{-5}i}{t^9}-
\frac{7.000697979\,10^{-9}}{t^{10}}+\frac{6.993000420\,10^{-8} i}{t^{11}}-
}\\
{\textstyle
\frac{2.100179772\,10^{-10}}{t^{12}}+
\frac{2.097780120\,10^{-9} i}{t^{13}}-\frac{6.602640528\,10^{-12}}{t^{14}}+\frac{6.601320792\,10^{-11} i}{t^{15}}}
\end{multline*}
\begin{multline*}
{\textstyle
\bigl(a_-^{(15)}(t)\bigr)_{22}=
1.+ \frac{5.003000600 i}{t} +\frac{0.9998998599\,10^{-2}}{t^2} -
\frac{0.9999999800\,10^{-1} i}{t^3} +}\\
{\textstyle
\frac{0.9997997799\,10^{-4}}{t^4}-\frac{0.9997998999\,10^{-3} i}{t^5}+\frac{0.1999499460\,10^{-5}}{t^6}-\frac{0.1999399700\,10^{-4} i}{t^7}+}\\
{\textstyle
\frac{4.998598479\,10^{-8}}{t^8}-\frac{4.998199079\,10^{-7} i}{t^9}+\frac{1.399579540\,10^{-9}}{t^{10}} -\frac{1.399439708\,10^{-8} i}{t^{11}}+}\\
\textstyle{\frac{4.198679400\,10^{-11}}{t^{12}} -\frac{4.198199904\,10^{-10} i}{t^{13}}
+\frac{1.319999894\,10^{-12}}{t^{14}}-\frac{1.320264053\,10^{-11} i}{t^{15}} }
\end{multline*}

\begin{table}[H]
{\scriptsize
\begin{center}
\begin{tabular}{|c|c|c|c|c|c|c|}
\hline
$N$&$\delta_+^{(N)}$&$\delta_-^{(N)}$&$\|a_+^{(40)}-a_+^{(N)}\|$
&$\|a_-^{(40)}-a_-^{(N)}\|$ &$\|a_+^{(N)}-a_+^{(N-1)}\|$&$\|a_-^{N}-a_-^{(N-1)}\|$\\
\hline
6& &948.109&$6.779868266\,10^{-6}$ &$1.398370474\,10^{-4}$&$2.673\,10^{-4}$&$5.794\,10^{-9}$\\
\hline
7& &144.467&$6.779739474\,10^{-6}$ &$3.505946296\,10^{-6}$&$1.537\,10^{-10}$&$1.365\,10^{-4}$\\
\hline
8& &27.580&$1.900289452\,10^{-7}$ &$3.505944083\,10^{-6}$&$6.684\,10^{-6}$&$3.629\,10^{-12}$\\
\hline
9& &5.466&$1.900288549\,10^{-7}$ &$9.835458009\,10^{-8}$&$1.075\,10^{-13}$&$3.414\,10^{-6}$\\
\hline
10& &1.091&$5.705084618\,10^{-9}$ &$9.835457840\,10^{-8}$&$1.871\,10^{-7}$&$2.849\,10^{-15}$\\
\hline
11&1.517&$2.182\,10^{-1}$&$5.705084542\,10^{-9}$&$2.954683944\,10^{-9}$&$9.046\,10^{-17}$&$9.559\,10^{-8}$\\
\hline
12&$1.760\,10^{-1}$&$4.363\,10^{-2}$&$1.794023946\,10^{-10}$&$2.954683942\,10^{-9}$&$5.614\,10^{-9}$&$2.567\,10^{-18}$\\
\hline
13&$3.248\,10^{-2}$&$8.726\,10^{-3}$&$1.794023946\,10^{-10}$&$9.295715063\,10^{-11}$&$8.538\,10^{-20}$&$2.868\,10^{-9}$\\
\hline
14&$6.397\,10^{-3}$&$1.745\,10^{-3}$&$5.833104901\,10^{-12}$&$9.295715063\,10^{-11}$&$1.764\,10^{-10}$&$2.538\,10^{-21}$\\
\hline
15&$1.275\,10^{-3}$&$3.490\,10^{-4}$&$5.833104901\,10^{-12}$&$3.023530855\,10^{-12}$&$8.727\,10^{-23}$&$9.013\,10^{-11}$\\
\hline
16&$2.549\,10^{-4}$&$6.981\,10^{-5}$&$1.945043487\,10^{-13}$&$3.023530855\,10^{-12}$&$5.734\,10^{-12}$&$2.682\,10^{-24}$\\
\hline
17&$5.098\,10^{-5}$&$1.396\,10^{-5}$&$1.945043487\,10^{-13}$&$1.008491079\,10^{-13}$&$9.460\,10^{-26}$&$2.929\,10^{-12}$\\
\hline
18&$1.019\,10^{-5}$&$2.792\,10^{-6}$&$6.615029120\,10^{-15}$&$1.008491079\,10^{-13}$&$1.911\,10^{-13}$&$2.982\,10^{-27}$\\
\hline
19&$2.039\,10^{-6}$&$5.585\,10^{-7}$&$6.615029120\,10^{-15}$&$3.430673589\,10^{-15}$&$1.073\,10^{-28}$&$9.764\,10^{-14}$\\
\hline
20&$4.078\,10^{-7}$&$1.117\,10^{-7}$&$2.285734191\,10^{-16}$&$3.430673589\,10^{-15}$&$6.499\,10^{-15}$&$3.448\,10^{-30}$\\
\hline
21&$8.157\,10^{-8}$ &$2.234\,10^{-8}$ &$2.285734191\,10^{-16}$ &$1.185661780\,10^{-16}$ &  &$3.320\,10^{-15}$\\
\hline
22&$1.631\,10^{-8}$ &$4.468\,10^{-9}$ &$8.001677628\,10^{-18}$ &$1.185661780\,10^{-16}$ &  &$4.116\,10^{-33}$\\
\hline
23&$3.263\,10^{-9}$ &$8.937\,10^{-10}$ &$8.001677628\,10^{-18}$ &$4.151357952\,10^{-18}$ &  &$1.147\,10^{-16}$\\
\hline
24&$6.525\,10^{-10}$ &$1.787\,10^{-10}$ &$2.831851035\,10^{-19}$ &$4.151357952\,10^{-18}$ &  &$5.044\,10^{-36}$\\
\hline
25&$1.305\,10^{-10}$ &$3.574\,10^{-11}$ &$2.831851035\,10^{-19}$ &$1.469410102\,10^{-19}$ &  &$4.014\,10^{-18}$\\
\hline
26&$2.610\,10^{-11}$ &$7.148\,10^{-12}$ &$1.011526608\,10^{-20}$ &$1.469410102\,10^{-19}$ &  &$6.317\,10^{-39}$\\
\hline
27&$5.220\,10^{-12}$ &$1.429\,10^{-12}$ &$1.011526608\,10^{-20}$ &$5.249343203\,10^{-21}$ &  &$1.420\,10^{-19}$\\
\hline
28&$2.044\,10^{-12}$ &$2.859\,10^{-13}$ &$3.64197255\,10^{-22}$ &$5.249343203\,10^{-21}$ &  &$8.059\,10^{-42}$\\
\hline
29&$2.088\,10^{-13}$ &$5.719\,10^{-14}$ &$3.64197255\,10^{-22}$ &$1.890220534\,10^{-22}$ &  &$5.072\,10^{-21}$\\
\hline
30&$4.176\,10^{-14}$ &$2.244\,10^{-14}$ &$1.320\,10^{-23}$ &$1.890220534\,10^{-22}$ &  &$1.045\,10^{-44}$\\
\hline
\end{tabular}
\end{center}
}
\caption{ Comparison $\delta_{\pm}^{(N)}$ with $\|a_{\pm}^{(40)}-a_{\pm}^{(N)}\|_W$ and the guaranteed accuracy of the approximate canonical factorisation $a(t)\approx a_-^{(N)}(t)a_+^{(N)}(t)$ for Example \ref{ex1A}.}\label{dopver1}
\end{table}

\noindent
{\bf Example 6.2}

\begin{table}[H]
{\scriptsize
\begin{center}
\begin{tabular}{|c|c|c|c|c|c|c|c|c|}
\hline
$N$&$\delta_N$&$\|(a_+^{(N)})^{-1}\|$&$\|(a_-^{(N)})^{-1}\|$&$q_N$&$\gamma_N$&$\|a_{30}-a_N\|_W$&$\|a_{40}-a_N\|_W$\\
\hline
1&55.659 & & & 333.955& $3.212 \,10^{7}$&1.292202089 &1.292202089\\
\hline
2&13.915 & & &243.509  & $1.859 \,10^{8}$&$4.603672092\,10^{-1}$&$4.603672092\,10^{-1}$\\
\hline
3&3.479 &13.00000000&26.38888889&1193.388&$2.539\,10^{11}$ &$9.653048042\,10^{-2}$&$9.653048042\,10^{-2}$\\
\hline
4&$8.697\,10^{-1}$&28.00539724&43.12791109&1050.407&$1.409\,10^{12}$ &$2.075037196 \,10^{-2}$&$2.075037196 \,10^{-2}$\\
\hline
5&$2.174\,10^{-1}$&106.2638793&193.5327650&4471.338&$2.315\,10^{15}$ &$3.112260817\,10^{-2}$&$3.112260817 \,10^{-3}$\\
\hline
6&$5.435\,10^{-2}$&106.0198599&193.0173811&1112.297&$1.423\,10^{14}$ &$5.052520022 \,10^{-4}$&$5.05252002 \,10^{-4}$\\
\hline
7&$1.359\,10^{-2}$&106.0270621&193.0175082&278.093&$8.897\,10^{12}$ &$5.588691252 \,10^{-5}$&$5.588691252\,10^{-5}$\\
\hline
8&$3.397\,10^{-3}$&106.0270621&193.0174816&69.521&$5.560\,10^{11}$ &$8.538922599 \,10^{-6}$&$8.538922599\,10^{-6}$\\
\hline
9&$8.493\,10^{-4}$&106.0230467&193.0174821&17.380 & $3.475\,10^{10}$&$6.805090541\,10^{-7}$&$6.805090541\,10^{-7}$\\
\hline
10&$2.123 \,10^{-4}$&106.0230244&193.0174825&4.345& $2.172\,10^{9}$&$1.488151391\,10^{-7}$&$1.488151391\,10^{-7}$\\
\hline
11&$5.308\,10^{-5}$&106.0230225&193.0174825&1.086&$1.357\,10^{8}$ &$7.911542747\,10^{-9}$&$7.911542747\,10^{-9}$\\
\hline
12&$1.327\,10^{-5}$&106.0230224&193.0174825&$2.716\,10^{-1}$& $8.484\,10^{6}$&$3.854190119\,10^{-9}$&$3.854190119\,10^{-9}$\\
\hline
13&$3.317 \,10^{-6}$&106.0230223&193.0174825&$6.789\,10^{-2}$& $5.303\,10^{5}$&$1.546768305\,10^{-10}$&$1.546768305\,10^{-10}$\\
\hline
14&$8.294\,10^{-7}$&106.0230223&193.0174825&$1.697\,10^{-2}$&$3.314\,10^{4}$ &$1.322654352\,10^{-10}$&$1.322654352\,10^{-10}$\\
\hline
15&$2.073\,10^{-7}$&106.0230223&193.0174825&$4.243\,10^{-3}$ &$2.071 \,10^{3}$ &$5.000553058\,10^{-12}$&$5.000553058\,10^{-12}$\\
\hline
16& $5.184\,10^{-8}$&106.0230223&193.0174825&$1.061\,10^{-3}$& $1.295 \,10^{2}$&$4.906795682\,10^{-12}$&$4.906795683\,10^{-12}$\\
\hline
17&$1.296\,10^{-8}$&106.0230223&193.0174825&$2.652\,10^{-4}$& 8.091&$1.854978316\,10^{-13}$&$1.854978320\,10^{-13}$\\
\hline
18&$3.240\,10^{-9}$&106.0230223&193.0174825&$6.630\,10^{-5}$&$5.057\,10^{-1}$  &$1.851904638\,10^{-13}$&$1.851904642\,10^{-13}$\\
\hline
19&$8.099\,10^{-10}$&106.0230223&193.0174825&$1.657\,10^{-5}$&$3.161\,10^{-2}$ &$7.038266444\,10^{-15}$&$7.038266829\,10^{-15}$\\
\hline
20&$2.025\,10^{-10}$&106.0230223&193.0174825&$4.144\,10^{-6}$&$1.975\,10^{-3}$ &$7.037455514\,10^{-15}$&$7.037455898\,10^{-15}$\\
\hline
21&$5.062\,10^{-11}$&106.0230223&193.0174825&$1.036\,10^{-6}$&$1.235\,10^{-4}$ &$2.687290630\,10^{-16}$&$2.687294473\,10^{-16}$\\
\hline
22&$1.265\,10^{-11}$&106.0230223&193.0174825&$2.590\,10^{-7}$&$7.716\,10^{-6}$ &$2.687273040\,10^{-16}$&$2.687276884\,10^{-16}$\\
\hline
23&$3.164\,10^{-12}$&106.0230223&193.0174825&$6.475\,10^{-7}$&$4.823\,10^{-7}$ &$1.030169742\,10^{-17}$&$1.030208176\,10^{-17}$\\
\hline
24&$7.910\,10^{-13}$&106.0230223&193.0174825&$1.619\,10^{-8}$&$3.014\,10^{-8}$ &$1.030169423\,10^{-17}$&$1.030207857\,10^{-17}$\\
\hline
25&$1.977\,10^{-13}$&106.0230223&193.0174825&$4.047\,10^{-9}$& $1.884\,10^{-9}$&$3.958771025\,10^{-19}$&$3.962619441\,10^{-19}$\\
\hline
26&$4.943\,10^{-14}$&106.0230223&193.0174825&$1.012\,10^{-9}$&$1.177\,10^{-10}$ &$3.958770976\,10^{-19}$&$3.962619392\,10^{-19}$\\
\hline
27& $1.236\,10^{-14}$&106.0230223&193.0174825&$2.529\,10^{-10}$& $7.359\,10^{-12}$&$1.491648249\,10^{-20}$&$1.528533066\,10^{-20}$\\
\hline
28&$3.090\,10^{-15}$&106.0230223&193.0174825&$6.323\,10^{-11}$&$4.599\,10^{-13}$ &$1.491648248\,10^{-20}$&$1.528533066\,10^{-20}$\\
\hline
29&$7.724\,10^{-16}$&106.0230223&193.0174825&$1.581\,10^{-11}$&$2.875\,10^{-14}$ &$7.617715140\,10^{-33}$&$5.910646257\,10^{-22}$\\
\hline
30&$1.931\,10^{-16}$ &106.0230223&193.0174825&$3.952\,10^{-12}$&$1.797\,10^{-15}$ & 0 &$5.911\,10^{-22}$\\
\hline
31&$4.828\,10^{-17}$ & & &  & &  &$2.290481164\,10^{-23}$\\
\hline
32&$1.207\,10^{-17}$ & & &  & &  &$2.290481164\,10^{-23}$\\
\hline
33&$3.017\,10^{-18}$  & & & & &  &$8.892504277\,10^{-25}$\\
\hline
34&$7.543\,10^{-19}$ & & &  & &  &$8.892504277\,10^{-25}$\\
\hline
35&$1.885\,10^{-19}$ & & &  & &  & $3.455038998\,10^{-26}$\\
\hline
36&$4.714\,10^{-20}$  & & & & &  & $3.455038998\,10^{-26}$\\
\hline
37& $1.179\,10^{-20}$ & & & & &  & $1.314236521\,10^{-27}$\\
\hline
38&$2.946\,10^{-21}$ & & &  & &  & $1.314236521\,10^{-27}$\\
\hline
39&$7.366\,10^{-22}$  & & & & &  & $2.476508000\,10^{-48}$\\
\hline
40&$1.842\,10^{-22}$ & & &  & &  & 0\\
\hline
\end{tabular}
\end{center}
}
\caption{Verification of the criterion $q_N<1$ and verification of accuracy of the estimate (\ref{estdelta}) for Example \ref{ex2}.}\label{accdeltaEx2}
\end{table}

\begin{table}[H]
{\scriptsize
\begin{center}
\begin{tabular}{|c|c|c|c|c|}
\hline
$N$&$\delta_+^{(N)}$&$\delta_-^{(N)}$&$\|a_+^{(40)}-a_+^{(N)}\|$
&$\|a_-^{(40)}-a_-^{(N)}\|$\\
\hline
12& &$5.164\,10^{4}$&$4.878\,10^{-8}$ &$5.206370288\,10^{-11}$ \\
\hline
13& &$1.009\,10^{4}$&$3.494\,10^{-9}$ &$1.934759225\,10^{-12}$ \\
\hline
14& &$2.391\,10^{3}$&$2.334\,10^{-10}$ &$1.933162965\,10^{-12}$ \\
\hline
15& &$5.902\,10^{2}$&$1.462\,10^{-11}$ &$7.249988247\,10^{-14}$ \\
\hline
16& &$1.471\,10^{2}$&$8.617\,10^{-13}$ &$7.249964290\,10^{-14}$ \\
\hline
17& &$3.674\,10^{1}$&$4.795\,10^{-14}$ &$2.73906237\,10^{-15}$ \\
\hline
18&$1.022\,10^{5}$&9.184&$2.527\,10^{-15}$ &$2.739062342\,10^{-15}$ \\
\hline
19&$1.305\,10^{4}$&2.296&$1.265\,10^{-16}$ &$1.040901891\,10^{-16}$ \\
\hline
20&$3.165\,10^{3}$&$5.740\,10^{-1}$&$6.033\,10^{-18}$ &$1.040901891\,10^{-16}$ \\
\hline
21&$7.898\,10^{2}$&$1.435\,10^{-1}$ &$2.745\,10^{-19}$ &$3.974538060\,10^{-18}$ \\
\hline
22&$1.974\,10^{2}$& $3.587\,10^{-2}$ &$1.195\,10^{-20}$ &$3.974538060\,10^{-18}$ \\
\hline
23&$4.935\,10^{1}$&$8.968\,10^{-3}$&$4.983\,10^{-22}$ &$1.523633108\,10^{-19}$ \\
\hline
24&$1.234\,10^{1}$&$2.242\,10^{-3}$&$1.995\,10^{-23}$ &$1.523633108\,10^{-19}$ \\
\hline
25&3.085&$5.605\,10^{-4}$ &$7.680\,10^{-25}$ &$5.860325899\,10^{-21}$ \\
\hline
26&$7.711\,10^{-1}$&$1.401\,10^{-4}$&$2.846\,10^{-26}$ &$5.860325899\,10^{-21}$ \\
\hline
27&$1.928\,10^{-1}$&$3.503\,10^{-5}$ &$1.017\,10^{-27}$ &$2.260477723\,10^{-22}$ \\
\hline
28&$4.820\,10^{-2}$ &$8.758\,10^{-6}$ &$3.510\,10^{-29}$ &$2.260477723\,10^{-22}$ \\
\hline
29&$1.205\,10^{-2}$ &$2.189\,10^{-6}$ &$1.171\,10^{-30}$ &$8.740726716\,10^{-24}$ \\
\hline
30&$3.012\,10^{-3}$&$5.474\,10^{-7}$ &$3.779\,10^{-32}$ & $8.740726716\,10^{-24}$\\
\hline
\end{tabular}
\end{center}
}
\caption{Comparison $\delta_{\pm}^{(N)}$ with 
$\|a_{\pm}^{(40)}-a_{\pm}^{(N)}\|_W$ for Example \ref{ex2}. }\label{compEx2}
\end{table}

\noindent
{\bf Example 6.3}

%NEW variant 
\begin{table}[H]
{\scriptsize
\begin{center}
\begin{tabular}{|c|c|c|c|c|c|c|}
\hline
$N$&$\delta_N$&$\|(a_+^{(N)})^{-1}\|$&
$\|(a_-^{(N)})^{-1}\|$&$q_N$&$\|a_{30}-a_N\|_W$&$\|a_{40}-a_N\|_W$\\
\hline
1&$4.464\,10^{5}$  & & & &2.031 &2.031\\
\hline
2&$4.464\,10^{4}$ & & & & $5.062\,10^{-1}$ &$5.062\,10^{-1}$\\
\hline
3&$4.464\,10^{3}$ & & & & $1.095\,10^{-1}$ &$1.095\,10^{-1}$ \\
\hline
4&446.390&$7.488774053\,10^{5}$&$3.49502\,10^{5}$&$2.337\,10^{14}$&$2.070\,10^{-2}$ &$2.070\,10^{-2}$\\
\hline
5&44.639&$9.993003878\,10^{6}$&$4.548002871\,10^{6}$&$4.057\,10^{15}$&$3.366\,10^{-3}$ &$3.366\,10^{-3}$ \\
\hline
6&4.464&$2.258318561\,10^{7}$&$1.151727289\,10^{7}$&$2.322\,10^{15}$&$4.745\,10^{-4}$&$4.745\,10^{-4}$\\
\hline
7&$4.464\,10^{-1}$&$2.484150034\,10^{8}$&$1.266899999\,10^{8}$&$2.810\,10^{16}$&$5.878\,10^{-5}$&$5.878\,10^{-5}$\\
\hline
8&$4.464\,10^{-2}$&$2.484150037\,10^{8}$&$1.266900000\,10^{8}$&$2.810\,10^{15}$&$6.486\,10^{-6}$&$6.486\,10^{-6}$\\
\hline
9&$4.464\,10^{-3}$&$2.484150037\,10^{8}$&$1.266900000\,10^{8}$&$2.810\,10^{14}$&$6.451\,10^{-7}$ &$6.451\,10^{-7}$ \\
\hline
10&$4.464\,10^{-4}$&$2.484150037\,10^{8}$&$1.266900000\,10^{8}$&$2.810\,10^{13}$&$5.838\,10^{-8}$ &$5.838\,10^{-8}$\\
\hline
11&$4.464\,10^{-5}$&$2.484150037\,10^{8}$&$1.266900000\,10^{8}$&$2.810\,10^{12}$&$4.846\,10^{-9}$ &$4.846\,10^{-9}$ \\
\hline
12&$4.464\,10^{-6}$&$2.484150037\,10^{8}$&$1.266900000\,10^{8}$&$2.810\,10^{11}$&$3.716\,10^{-10}$ &$3.716\,10^{-10}$ \\
\hline
13&$4.464\,10^{-7}$&$2.484150037\,10^{8}$&$1.266900000\,10^{8}$&$2.810\,10^{10}$&$2.647\,10^{-11}$ &$2.647\,10^{-11}$ \\
\hline
14&$4.464\,10^{-8}$&$2.484150037\,10^{8}$&$1.266900000\,10^{8}$&$2.810\,10^{9}$&$1.760\,10^{-12}$ &$1.760\,10^{-12}$ \\
\hline
15&$4.464\,10^{-9}$&$2.484150037\,10^{8}$&$1.266900000\,10^{8}$&$2.810\,10^{8}$ &$1.098\,10^{-13}$ &$1.098\,10^{-13}$ \\
\hline
16&$4.464\,10^{-10}$&$2.484150037\,10^{8}$&$1.266900000\,10^{8}$&$2.810\,10^{7}$&$6.445\,10^{-15}$&$6.445\,10^{-15}$ \\
\hline
17&$4.464\,10^{-11}$&$2.484150037\,10^{8}$&$1.266900000\,10^{8}$&$2.810\,10^{6}$&$3.574\,10^{-16}$ &$3.574\,10^{-16}$ \\
\hline
18&$4.464\,10^{-12}$&$2.484150037\,10^{8}$&$1.266900000\,10^{8}$&$2.810\,10^{5}$&$1.878\,10^{-17}$ &$1.878\,10^{-17}$ \\
\hline
19&$4.464\,10^{-13}$&$2.484150037\,10^{8}$&$1.266900000\,10^{8}$&$2.810\,10^{4}$&$9.379\,10^{-19}$ &$9.379\,10^{-19}$ \\
\hline
20&$4.464\,10^{-14}$&$2.484150037\,10^{8}$&$1.266900000\,10^{8}$&$2.810\,10^{3}$&$4.461\,10^{-20}$ &$4.461\,10^{-20}$\\
\hline
21&$4.464\,10^{-15}$&$2.484150037\,10^{8}$&$1.266900000\,10^{8}$&$2.810\,10^{2}$&$2.025\,10^{-21}$ &$2.025\,10^{-21}$ \\
\hline
22&$4.464\,10^{-16}$&$2.484150037\,10^{8}$&$1.266900000\,10^{8}$&$2.810\,10^{1}$&$8.796\,10^{-23}$ &$8.796\,10^{-23}$\\
\hline
23&$4.464\,10^{-17}$&$2.484150037\,10^{8}$&$1.266900000\,10^{8}$&2.810&$3.662\,10^{-24}$ &$3.662\,10^{-24}$\\
\hline
24&$4.464\,10^{-18}$&$2.484150037\,10^{8}$&$1.266900000\,10^{8}$&$2.810\,10^{-1}$&$1.463\,10^{-25}$ &$1.463\,10^{-25}$ \\
\hline
25&$4.464\,10^{-19}$&$2.484150037\,10^{8}$&$1.266900000\,10^{8}$&$2.810\,10^{-2}$&$5.624\,10^{-27}$ &$5.624\,10^{-27}$\\
\hline
26&$4.464\,10^{-20}$&$2.484150037\,10^{8}$&$1.266900000\,10^{8}$&$2.810\,10^{-3}$&$2.081\,10^{-28}$ &$2.081\,10^{-28}$ \\
\hline
27&$4.464\,10^{-21}$&$2.484150037\,10^{8}$&$1.266900000\,10^{8}$&$2.810\,10^{-4}$&$7.428\,10^{-30}$ &$7.428\,10^{-30}$ \\
\hline
28&$4.464\,10^{-22}$&$2.484150037\,10^{8}$&$1.266900000\,10^{8}$&$2.810\,10^{-5}$&$2.558\,10^{-31}$ &$2.560\,10^{-31}$ \\
\hline
29&$4.464\,10^{-23}$&$2.484150037\,10^{8}$&$1.266900000\,10^{8}$&$2.810\,10^{-6}$&$8.375\,10^{-33}$ &$8.528\,10^{-33}$\\
\hline
30&$4.464\,10^{-24}$&$2.484150037\,10^{8}$&$1.266900000\,10^{8}$&$2.810\,10^{-7}$&0 & $2.749\,10^{-34}$\\
\hline
31& & & & & & $8.586\,10^{-36}$ \\
\hline
32& & & & & & $2.601\,10^{-37}$ \\
\hline
33& & & & & & $7.645\,10^{-39}$ \\
\hline
34& & & & & & $2.183\,10^{-40}$ \\
\hline
35& & & & & & $6.062\,10^{-42}$ \\
\hline
36& & & & & & $1.638\,10^{-43}$ \\
\hline
37& & & & & & $4.308\,10^{-45}$ \\
\hline
38& & & & & & $1.104\,10^{-46}$ \\
\hline
39& & & & & & $2.722\,10^{-48}$ \\
\hline
40& & & & & & 0 \\
\hline
\end{tabular}
\end{center}
}
\caption{Verification of the criterion $q_N<1$  and verification of accuracy of the estimate (\ref{estdelta}) for Example \ref{ex3}.}\label{crit3norm}
\end{table}

\end{document}